\documentclass{amsart}
\usepackage{amsmath,amsthm,amssymb,amsfonts,amscd}

\newcommand \p {\mathbb{P}}
\newcommand \lex {\ensuremath{\mathrm{lex}}}

\newcommand \sat {\textup{sat}}

\newcommand \mac[1] {\langle #1 \rangle}

\newcommand \reg {\textup{reg}}
\newcommand \codim {\textup{codim}}

\newcommand \N {\mathbb{N}}
\newcommand \X {\mathbb{X}}
\newcommand \Q {\mathbb{Q}}
\newcommand \Y {\mathbb{Y}}
\newcommand \Z {\mathbb{Z}}
\def\ds{\displaystyle}

\theoremstyle{plain} 
\newtheorem{Thm}{\bf Theorem}[section]
\newtheorem{Prop}[Thm]{\bf Proposition}
\newtheorem{Coro}[Thm]{\bf Corollary}
\newtheorem{Lem}[Thm]{\bf Lemma}

\theoremstyle{definition}
\newtheorem{Def}[Thm]{\bf Definition}
\newtheorem{Ex}[Thm]{\bf Example}
\newtheorem{Remk}[Thm]{\bf Remark}

\newtheorem{ques}[Thm]{\bf Question}

\numberwithin{equation}{section}

\numberwithin{equation}{section}

\begin{document}

\title{The Gotzmann Coefficients of Hilbert Functions}
\author[J. Ahn]{Jeaman Ahn${}^\dag$}
\thanks{${}^\dag$This paper was supported by a grant from Kongju National University in 2008.}
\address[J. Ahn]{Department of Mathematics Education, Kongju National University, 182, Shinkwan-dong, Kongju, Chungnam 314-701,  Republic of Korea}
\email{jeamanahn@kongju.ac.kr}

\author[A.V. Geramita]{Anthony V. Geramita}
\address[A.V. Geramita]{Department of Mathematics, Queen's University, Kingston, Ontario, Canada, K7L 3N6 and Dipartimento di Mathematica Universit\'a di Genova, Genova, Italia}
\email{tony@mast.queensu.ca \\ geramita@dima.unige.it}

\author[Y.S. Shin]{Yong Su Shin${}^*$}
\thanks{${}^*$This paper was supported by a grant from Sungshin Women's University in 2007.}
\address[Y.S. Shin]{Department of Mathematics, Sungshin Women's University,  Seoul, 136-742, Republic of Korea}
\email{ysshin@sungshin.ac.kr}

\begin{abstract}
In this paper we investigate some algebraic and geometric consequences which arise from an extremal bound on the Hilbert function of the general hyperplane section of a variety (Green's Hyperplane Restriction Theorem).  These geometric consequences improve some results in this direction first given by Green and extend others by Bigatti, Geramita, and Migliore.

Other applications of our detailed investigation of how the Hilbert polynomial is written as a sum of binomials, are to conditions that must be satisfied by a polynomial if it is to be the Hilbert polynomial of a non-degenerate integral subscheme of $\mathbb P^n$ (a problem posed by R. Stanley).  We also give some new restrictions on the Hilbert function of a zero dimensional reduced scheme with the Uniform Position Property.
\end{abstract}

\subjclass[2000]{Primary:13A02, 13A15; Secondary:14H45, 14H50}
\keywords{Hilbert functions,  Hilbert polynomials, Gotzmann Coefficients, Gotzmann  numbers, Integral Subschems, Uniform Position Property.}

\maketitle

\section{Introduction}

In his seminal work on Hilbert functions of standard graded algebras, Macaulay \cite{M} discovered the rule governing the growth of such functions.  He expressed that rule in terms of certain expansions of the values of these functions by binomial numbers.  Indeed, part of the genius of Macaulay's approach is in the discovery of this uniform approach to the problem of understanding the nature of these functions.

After an initial hiatus of a few decades, the depth and value of Macaulay's approach was again appreciated.  Indeed, in the last half century, the importance of the Hilbert function (and Hilbert polynomial) to the study of algebraic varieties and to commutative rings is hard to overestimate.  An integral part of that appreciation of the importance of the Hilbert function has been accompanied by several reappraisals of Macaulay's original proof and significant refinements have been made to that original argument.  Most notable among these are the work of Stanley~\cite{St}, Green~\cite{G1}, and Gotzmann~\cite{Go}.

In fact, Green's approach to Macaulay's Theorem included a brand new element --- a comparison between the Hilbert function of a variety and that of its general hyperplane section (see Theorem~\ref{T:201} below).

It is well known that the Hilbert function and Hilbert polynomial of an embedded algebraic variety, although being natural {\it algebraic} invariants associated to the coordinate ring of the embedded variety, also carry significant geometric information about the variety - some of the information being connected to the embedding, like the degree of the variety, while other information is more intrinsic (i.e., does not depend on the embedding) like the dimension and genus of the variety.

A great deal of research has been conducted with the aim of extracting other such geometric information from the Hilbert function and Hilbert polynomial.  The papers \cite{AC, BGM, EH1, G1, Go, I_K, Mi, Preser} give a small sample of the kinds of investigations that have been carried out in this direction. The book \cite{KR} (especially Chapter 5) is an excellent ``one-stop" view of most of this work, including (new and short) proofs
of both Macaulay¡¯s and Green¡¯s Theorems.

This paper falls into that tradition of trying to understand the geometric consequences of certain behavior of the Hilbert function.  Unlike most earlier investigations in this direction (but present already in the work of Gotzmann \cite{Go}, Iarrobino-Kleiman \cite{I_K}, Ahn--Cho \cite{AC}, and implicitly in Kreuzer-Robbiano~\cite{KR}, Section 5.5) we concentrate not only on the values of the Hilbert function but exactly how those values are expressed by Macaulay's original considerations.

For example, if $\X$ is an irreducible variety in $\p^n$ then the number of binomial summands that (eventually) appear in the Hilbert function of $\X$ is an invariant of $\X$ (see Definition \ref{Gotzmann}) denoted $G(\X)$ and called the {\it Gotzmann number of $\X$}.  It is not difficult to show that $\deg\X \leq G(\X)$.  We characterize the varieties $\X$ for which this inequality is an equality (see Theorem~\ref{T:311}).  This result follows from a detailed investigation of precisely when the inequality in Green's Hyperplane Restriction Theorem is an equality.

In fact, if we denote by $M(\X)$ the least integer such that the inequality in Green's Theorem is an equality for all $d \geq M(\X)$ then one easily sees that $M(\X) \leq G(\X) + 1$.  We improve this to show that $M(\X) \leq G(\X)$ (see Proposition~\ref{P:301}) and then go on to show that $M(\X) = G(\X)$ or $M(\X) = 1$ (see Proposition~\ref{P:306}).  Connecting this to our earlier geometric discussion we show that if $\X$ is a reduced, equidimensional closed subscheme of $\p^n$ then either $G(\X) = \deg(\X)$ or $G(\X) = M(\X)$.

Continuing with our investigation of Macaulay's way of writing the Hilbert polynomial of a scheme, we prove that if $\X$ is a non-degenerate reduced equidimensional closed subscheme of $\p^n$ of codimension $\geq 2$ then Macaulay's decomposition of the Hilbert polynomial must satisfy certain properties (see Theorem \ref{T:404}).  We use these observations to exhibit new restrictions on the Hilbert function of a set of points in $\p^n$ with the Uniform Position Property (see Theorem \ref{T:506}).

The paper is organized in the following way: in Section 2 we recall the essential parts of Macaulay's Theorem on the growth of the Hilbert function of a standard graded $k$-algebra, $k$ a field, usually infinite.  If more conditions on $k$ are required we shall state so at the relevant place.  In this section one finds the definitions of the Gotzmann number and the Gotzmann coefficients.  We also recall Green's Hyperplane Restriction Theorem in this section.  After that we calculate the various invariants we have introduced in a few special situations.

In Section 3 we investigate the possibility of equality in Green's Hyperplane Restriction Theorem and study this condition in detail.  This is the technical heart of the paper.  We also give a few of our main consequences of this investigation in this section.

In Section 4 we investigate the nature of the Gotzmann coefficients for reduced equidimensional closed subschemes of $\p^n$.  These are rather delicate invariants and we show that, in the relevant range (and apart from a completely describable collection of exceptions) these coefficients are never zero.

This investigation sets us up for the discussion in the final section, Section 5, on the Hilbert function of points with the Uniform Position Property.

\section{\sc Preliminaries}\label{Preliminaries}

Many of the preliminaries we will discuss in this section are based on the fundamental work of Macaulay (and subsequently that of G. Gotzmann) which describe the growth of the dimensions of the homogeneous summands of a standard graded $k$-algebra.

Recall that if $d >0$ and $c>0$ are two integers then the $d$-{\it binomial expansion of $c$} is the unique expression
\begin{equation}\label{EQ:201}
c=\binom {k_d}{d}+\binom {k_{d-1}}{d-1}+\cdots +\binom
{k_\delta}{\delta}   
\end{equation}
where $k_d>k_{d-1}>\cdots >k_\delta \ge \delta >0$.

\begin{Ex}
The $4$-binomial expansion of 27 is:
$$
27 = \binom{6} {4} + \binom{5} {3} + \binom{2} {2} + \binom{1} {1} .
$$
\end{Ex}

An equivalent way to describe the $d$-binomial expansion of $c$ is to construct what we define to be the {\it $d^{th}$ Macaulay difference set of $c$}, i.e., the tuple
$$
M_d(c) = (k_d - d, k_{d-1} -(d-1), \dots , k_\delta - \delta )
$$
(where we use the notation of equation~\eqref{EQ:201} above).  Notice that this tuple has the property
$$
k_d - d \geq k_{d-1} -(d-1) \geq \cdots \geq k_\delta - \delta \geq 0 .
$$
\begin{Ex}
From the example above we have
$$
M_4(27) = (2,2,0,0).
$$
\end{Ex}

We define the {\it length of $M_d(c)$} to be the number of entries in $M_d(c)$ (e.g., the length of $M_4(27)$ is $4$).

If we are given a tuple of $d - \delta +1 \leq d $ integers, say $(a_d, a_{d-1}, \dots ,a_\delta)$, such that $ a_d \geq a_{d-1} \geq \cdots \geq a_\delta \geq 0$ then that tuple is $M_d(c)$ for the integer
$$
c = \binom{d + a_d}{d} + \binom{(d-1) + a_{d-1}}{d-1} + \cdots + \binom{\delta + a_\delta} {\delta} .
$$

We will see, in the ensuing sections, that when we construct these multisets for the various values of the Hilbert function of an algebraic variety $\X \subset \p ^n$ then the entries of these multisets will signal subtle information about geometric properties of $\X$.

A fundamental result of Macaulay highlights the importance of the following functions, which are defined for every integer $d > 0$.  These functions from $\mathbb N$ to $\mathbb N$ are now called {\it Macaulay's functions}.  They are denoted $ - ^{\langle d\rangle}$ (and referred to as \lq\lq {\it{upper pointy bracket d}}"), and are defined as follows: if $c > 0$ and the $d$-binomial expansion of $c$ is as above in equation~\eqref{EQ:201}, then
$$
c^{\mac d}:= \binom {k_d + 1}{d + 1}+\binom {k_{d-1} + 1}{d}+\cdots +\binom {k_\delta + 1}{\delta + 1}.
$$

\medskip
\begin{Remk}\label{nochange} Notice that the $d^{th}$ Macaulay difference set of $c$ and the $(d+1)^{st}$ Macaulay difference set of $c^{\langle d \rangle}$ are the same.
\end{Remk}
\medskip

Another (similar) collection of functions was introduced and exploited by Green \cite{G1}.  They are denoted $-_{\mac d}$, (and referred to as \lq\lq {\it{lower pointy bracket d}}"), and defined by
$$
c_{\mac d}:= \binom {k_d -1}{d}+\binom {k_{d-1}-1}{d-1}+\cdots +\binom
{k_\delta -1}{\delta},
$$
(where the convention is that $\binom{i}{j} = 0$, when $i < j$).  We will call these functions {\it Green's functions}.

\medskip  Now, if $I = \bigoplus_{j \geq 0}I_j$ is a homogeneous ideal of $R = k[ x_0, x_1, \ldots , x_n]$ ($k$ algebraically closed of characteristic 0), then the graded ring
$$
A = \bigoplus_{j \geq 0}A_j = \bigoplus_{j \geq 0} \left( {R_j \over I_j} \right)
$$
is a {\it standard graded $k$-algebra}.

The function $H(A,-):\N \rightarrow \N$ defined by
$$
H(A,j) = \dim_kA_j = \dim_kR_j - \dim_k I_j
$$
is called the {\it Hilbert function of the ring $A$}.  It is well known that there is a polynomial $P(z) \in \Q[z]$ (called the {\it Hilbert polynomial of $A$}) with the property that, for all integers $t \gg 0$, $H(A,t) = P(t)$.  I.e., the eventual behavior of the function $H(A,-)$ is that of a polynomial with rational coefficients.  Moreover, the degree of the polynomial $P(z)$ is one less than the Krull dimension of $R/I$.

If $\X$ is a closed subscheme of $\p^n$ with defining homogeneous ideal $I = I_\X$ then we will often use $S_\X = R/I_\X$, instead of $A = R/I$, to denote the homogeneous coordinate ring of $\X$.  In this case the function $H(A,-)$ will be denoted either $H(S_\X,-)$ or simply $H_\X(-)$, and called {\it the Hilbert function of $\X$}.  In like fashion, the Hilbert polynomial of $A$ is usually referred to as {\it the Hilbert polynomial of $\X$}.  Since the Hilbert function and Hilbert polynomial of $\X$ encode a great deal of interesting geometric information about $\X$, these objects have long been the subject of intensive study.

The importance of both Macaulay's and Green's functions are a consequence of the fact that they give significant information about Hilbert functions of standard graded $k$-algebras and hence about the geometry of projective varieties.  We now recall the exact roles of these functions.

\begin{Thm}[\cite{G1}, Chapter 5 \cite{KR}, \cite{M}, \cite{St}]\label{T:201}
 Let $I\subset R$ be a homogeneous ideal and let $h \in R_1$ be a general linear form. If we set $A=R/I$ then, for all $d\geq 1$ we have the following inequalities.
\begin{enumerate}
\item [(a)] {\rm Macaulay's Theorem:} $H(A,d+1) \leq H(A,d)^{\mac d}$.
\item[(b)] {\rm Green's Hyperplane Restriction Theorem:} $H(A/hA,d) \leq H(A,d)_{\mac d}.$
\end{enumerate}
\end{Thm}

\medskip In view of Macaulay's and Green's Theorems, it is not surprising that we will often be discussing binomial expansions for various values of the Hilbert function of some graded algebra $A = R/I$ (often when it is the coordinate ring of a closed subscheme $\X$ of $\p^n$, in which case $I = I_\X$ and $A = S_\X = R/I_\X$).  In this case we want to use some different terminology to describe features of the binomial expansion.

\begin{Def}\label{dthGotzmann}  If $A$ is a standard graded $k$-algebra (or $A = S_\X = R/I_\X$ for $\X$ a closed subscheme of $\p^n$) and $c = H(A,d)$ (or $c = H_\X(d)$) then we will refer to the length of the $d^{th}$ Macaulay difference set of $c$ as the $d^{th}$ {\it Gotzmann persistence number of $A$ (or $\X$)}.

In the former case the $d^{th}$ Gotzmann persistence number will be denoted $G(A,d)$ while in the latter case it will be denoted $G(\X,d)$.

The number of elements in the multiset $M_d(H(A,d))$ (or $M_d(H_\X(d))$ which are equal to $\ell$ will be called the $\ell^{th}$ {\it Gotzmann coefficient of $H(A,d)$} (or of $H_\X(d)$) and denoted $C_\ell(A,d)$ (or $C_\ell(\X,d)$).
\end{Def}

\begin{Ex} \begin{itemize}
\item[(a)] Let $A = { {k[x_0,x_1,x_2]}/{(x_0^2, x_1^3, x_2^4)}}$.  Then $H(A,3) = 6$ and the 3-binomial expansion of 6 is
$$
6  = \binom{4}3 + \binom{2}2 + \binom{1}1.
$$
So, the third Macaulay difference set of $A$ is $M_3(H(A,3)) = (1,0,0)$.  Thus the $3^{rd}$ Gotzmann persistence number of $A$ is $G(A,3) = 3$.  Furthermore  $C_1(A,3) = 1$ and $C_0(A,3)= 2$.
\item[(b)] Let $\X$ be the rational normal curve in $\p^4$.  Then $H_\X(3) = 13$ and the 3-binomial expansion of 13 is
$$
13 = \binom{5}3 + \binom{3}2.
$$
So, the third Macaulay difference set of $\X$ is $M_3(H_\X(3)) = (2,1)$.  Thus $G(\X,3)=2$ is the $3^{rd}$ Gotzmann persistence number of $\X$.  The second Gotzmann coefficient of $H_\X(3)$ is $C_2(\X,3) = 1$, while the first Gotzmann coefficient of $H_\X(3)$ is  $C_1(\X,3) = 1$ and the zeroth Gotzmann coefficient of $H_X(3)$ is $C_0(\X,3) = 0$.
\end{itemize}
\end{Ex}

\medskip If, for a ring $A$ and an integer $d$, we have equality in Macaulay's Theorem then we say that {\it the Hilbert function of $A$ has maximal growth in degree $d$}.  The following Theorem (one of the principal results of Theorem 3.3 in \cite{AC} and proved independently in Section 5.5 in \cite{KR}) shows that maximal growth in degree $d$ is related to equality also in Green's bound.   More precisely:

\begin{Thm}[\cite{AC}, \cite{KR}]\label{T:202}
 Let $I$ be a homogeneous ideal in $R$ and let $A=R/I$. Let
 \begin{equation}\label{bieq}
 H(A,d+1)=\binom{(d+1)+a_{d+1}}{(d+1)}+ \cdots+\binom{\delta+a_{\delta}}{\delta}
\end{equation}
be the $(d+1)$-binomial expansion of $H(A,d+1)$.   Suppose that $d \geq \sat (I)$  {\rm (}where $\sat(I)$ denotes the saturation degree of $I$, i.e., the least degree $r$ for which $I$ and $I^{\sat}$ agree in all degrees $j \geq r${\rm )}.

 Then, the following statements are equivalent:
 \begin{enumerate}
 \item[(a)] $H(A,d+1)=H(A,d)^{\mac d}.$
 \item[(b)] $\delta>1$ and  $H(A/hA,d+1)=H(A,d+1)_{\mac {d+1}}$ for a general linear form $h$ in
 $A_1$.
 \end{enumerate}
 \end{Thm}

Note that (a) implies (b) is true without the condition $d\geq \sat (I)$. However, the condition \lq\lq$\delta > 1$" is needed in Theorem~\ref{T:202}, as the following corollary shows (see \cite{AC}).

 \begin{Coro}\label{C:203}
Under the hypotheses of Theorem \ref{T:202}; if
$\delta =1$ in equation $(\ref{bieq})$ and $H(A/hA,d+1)=H(A,d+1)_{\mac {d+1} }$, then
\[H(A,d)^{\mac d}=H(A,d+1)+a_2-a_1+1.\]
\end{Coro}

\vskip .5cm
Recall that a polynomial $p(z) \in \Q[z]$ is called a {\it numerical polynomial} if $p(a) \in \Z$ whenever $a \in \Z$.  For us, the most important examples of numerical polynomials are
$$
b_i(z) = \binom{z}{i}:= {{z(z-1)\cdots (z-(i-1))} \over i!} .
$$
So, $b_0(z) = 1,\ b_1(z) = z,\ b_2(z) = z(z-1)/2, \ldots $.  Notice that $b_d(z)$ is a polynomial of degree exactly $d$ and hence the $\{ b_i(z) \mid i = 0, 1,\dots \}$ are a vector space basis for $\Q[z]$ over $\Q$.

It is a classical theorem that a numerical polynomial $p(z)$, of degree $d$, can be written
$$
p(z) = \alpha_db_d(z) + \alpha_{d-1}b_{d-1}(z) + \cdots + \alpha_1b_1(z) + \alpha_0
$$
where the $\alpha_i$ are all in $\Z$.

Typical examples of numerical polynomials are the Hilbert polynomials of standard graded $k$-algebras and one of the main questions which we will consider in this paper concerns such polynomials. That question is:

\begin{ques}
Let $\X \subset \p^n$ be a nondegenerate integral (or sometimes just reduced) scheme with Hilbert polynomial $P_\X(z) \in \Q[z]$.  What can we say about $P_\X(z)$?
\end{ques}

An important ingredient in trying to answer this question is the concept of {\it maximal growth}, which we defined above. It is not difficult to see (it follows from Macaulay's Theorem and a good look at Pascal's triangle) that: given an ideal $I \subset R$ and $A = R/I$, there is an integer $d$ (which depends on $I$) such that for all $j \geq d$ we have maximal growth in degree $j$, i.e.,
$$
H(A, j+1) = H(A,j)^{\mac{j}},
$$
It follows from Remark~\ref{nochange} that for all $j \geq d$ the $j^{th}$ Macaulay difference set of $H(A,j)$ doesn't change.

In view of this last observation, the following definitions all make sense.

\begin{Def}\label{Gotzmann}

\begin{enumerate}
\item[(a)] If $I$ is a homogeneous ideal of $R$ and $A = R/I$ then the eventually constant Macaulay difference sets for the numbers $H(A,t)$ will be called the {\it Gotzmann difference set of $A$.}  Since the Hilbert polynomial of $A$, $P_{A}(z)$, eventually always takes on the values of $H(A,-)$, we also refer to the Gotzmann difference set of $A$ as the {\it Gotzmann difference set of} $P_{A}(z)$.
\item[(b)] The number of elements in the Gotzmann difference set of $A$ will be called the {\it Gotzmann number of $A$} and denoted $G(A)$.
\item[(c)]  The number of times that the integer $i$ appears in the Gotzmann difference set of $A$ will be denoted $C_i(A)$ and called the {\it $i^{th}$ Gotzmann coefficient of $A$}.
\item[(d)] If $I = I_\X$ is the homogeneous ideal of a closed subscheme $\X \subset \p^n$ then these objects will be referred to as the {\it Gotzmann difference set of $\X$},  {\it the Gotzmann number of $\X$}, and the {\it $i^{th}$ Gotzmann coefficient of $\X$} respectively (see Iarrobino-Kleiman \cite{I_K}).  We will write $G(\X)$ for the Gotzmann number of $\X$ and $C_i(\X)$ for the $i^{th}$ Gotzmann coefficient of $\X$.
\end{enumerate}
\end{Def}

\begin{Remk}\label{delta>1} Let $A = R/I$ be as above and suppose that $d$ is an integer with the property that we have maximal growth for the Hilbert function of $A$ in degree $j$ for all $j \geq d$.

Let $H(A,d+1)$ be as in equation~\ref{bieq}.  Let's suppose (for the moment) that $\delta = 1$ in that expression.  Since we are in the range where $H(A,-)$ has maximal growth, we can write
$$
H(A,d+2) = \binom{d+2+a_{d+1}}{d+2} + \cdots + \binom{\delta + 1 + a_{\delta}}{\delta + 1}
$$
and then rewrite it as
$$
H(A,d+2) = \binom{d+2+a^\prime_{d+2}}{d+2} + \cdots + \binom{(\delta + 1) + a^\prime_{\delta +1}}{\delta + 1}.
$$
But now $\delta + 1 > 1$ and so the condition of Theorem~\ref{T:202} b) on \lq\lq$\delta$" is now satisfied.
\end{Remk}

\medskip




There is another observation we can make when the Hilbert function of $A = R/I$ has maximal growth in degrees $j \geq d$. This observation shows us how to compute $P_A(z)$, the Hilbert polynomial of $A$, from the Gotzmann difference set of $A$.    To explain this, let $c = H(A,d)$ and write the $d$-binomial expansion of $c$ as
$$
c = \binom{k_d}{d} + \binom{k_{d-1}}{d-1} + \cdots + \binom{k_\delta}{\delta}, k_d > k_{d-1} > \cdots > k_\delta > 0.
$$
Now rewrite this, first as
$$
c = \binom{d+a_d}{d} + \binom{(d-1) + a_{d-1}}{d-1}+ \cdots + \binom{\delta + a_\delta}{\delta}, a_d \geq a_{d-1} \geq \cdots \geq a_\delta \geq 0
$$
and then as
$$
c = \binom{d+a_d}{a_d} + \binom{d-1+a_{d-1}}{a_{d-1}} + \cdots + \binom{d-(d-\delta) + a_\delta}{a_\delta}.
$$

Since we are assuming we have maximal growth in degree $d$, we have
$$
H(A, d+1) = \binom{d+1 + a_d}{d+1} + \binom{d+a_{d-1}}{d} + \cdots + \binom{d-(d-\delta) + 1 + a_\delta}{\delta + 1}
$$
which we can rewrite as
$$
H(A, d+1) = \binom{d+1 + a_d}{a_d} + \binom{d+a_{d-1}}{a_{d-1}}+\cdots + \binom{d-(d-\delta) + 1 + a_\delta}{a_\delta}.
$$
Note that $\{ a_d, a_{d-1}, \ldots , a_\delta \}$ is the $d^{th}$ Macaulay difference set of $H(A,d)$ and also of $H(A,d+1)$.

But now, consider the numerical polynomials
$$
b_{a_d}(z+a_d),\ \ \ b_{a_{d-1}}(z -1 +a_{d-1}),\ \ \  \ldots \ \ \ , b_{a_\delta}(z -(d-\delta) + a_\delta )
$$
and let
$$
L(z) = \sum _{i=\delta}^d b_{a_i}(z+a_i-(d-i)).
$$
Clearly $L(d) = c = H(A,d)$ and $L(d+1) = c^{\mac{d}} = H(A, d+1)$.  I.e., the polynomial $L(z)$ describes the value of the Hilbert function in both degrees $d$ and $d+1$.

We can obviously continue this argument for as long as the growth of $H(A,-)$ is maximal.  So, given our assumption that we have maximal growth in degree $j$ for all integers $j \geq d$, we obtain that $L(z) = P_A(z)$, the Hilbert polynomial of $A =R/I$.

Notice also that since we have only made a linear changes of variables on the polynomials $b_{a_i}(z)$, the polynomial $b_{a_i}(z+a_i-(d-i))$ is again a polynomial of degree $a_i$.

By standard results about Hilbert polynomials it then follows that if
$$(a_d, a_{d-1}, \ldots, a_\delta)$$ is the Gotzmann difference set of the ring $A = R/I$, then $a_d$ is one less than the Krull dimension of $A$.

In the case where $I = I_\X$ is the defining ideal of a closed subscheme $\X$ of dimension $r$ then $a_d = r$ and all the numbers in the Gotzmann difference set of $\X$ are $\leq r$.  In particular the $i^{th}$ Gotzmann coefficient of $\X$ can be non-zero only if $0\leq i \leq r$.

Since, as we noted above, the polynomials $\binom{z}{i}, i = 0,1,\ldots $ are a $\Z$-basis for the free $\Z$-module of numerical polynomials in $\Q[z]$ and since the degree of $\X$ is the integer coefficient of the term of highest degree when $P_\X(z)$ is written in terms of this basis, it follows that if $\mathcal G = (a_d, \ldots , a_\delta) $ is the Gotzmann difference set of $\X$ then the degree of $\X$ is nothing more than the number of elements in $\mathcal G$ equal to $r=a_d$ i.e.,
\begin{equation}\label{degree}
C_r(\X) = \deg \X.
\end{equation}

We can use these remarks to observe, for example, that if $\X$ is a finite set of $s$ points in $\p^n$ then the Gotzmann difference set of $\X$ is $\{ 0,0, \ldots ,0 \}$ where $C_0(\X) = G(\X) = s$.

Let's look at another example, this time when $\X$ is a curve.  We know, in this case, that all the elements in the Gotzmann difference set are either $0$ or $1$ and that the number of $1$'s is the degree of $\X$.

\begin{Ex}\label{cubics}
 \begin{enumerate}

\item[(a)] Let $\mathcal C$ be the rational normal curve in $\p^3$.  Since that curve has degree 3, we know that the Gotzmann difference set has  its first three entries equal to 1 since $3 = \deg \mathcal C = C_1(X)$ is the $1^{th}$ Gotzmann coefficient of $\mathcal C$.  The only question that remains is: what is the $0^{th}$ Gotzmann coefficient of $\mathcal C$, i.e., what is $C_0(\mathcal C)$?

Recall that the Hilbert function of the rational normal curve in $\p^3$ is given by the sequence
$$
1,\ 4,\ 7,\ 10,\ 3z+1,  \ldots
$$
where the notation gives the first few values of the Hilbert function and then the eventual behavior of the succeeding terms.

Since
$$
13 = \binom{5}{4} + \binom{4}{3} + \binom{3}{2} + \binom{1}{1}
$$
we can easily see that maximal growth begins in degree 4.   It follows that the Hilbert polynomial of the rational normal curve in $\p^3$ is
$$
P_{\mathcal C}(z) = 3z+1 = \binom{z+1}{1} + \binom{z}{1} + \binom{z-1}{1} + \binom{z-3}{0} .
$$
Thus, the Gotzmann difference set of $\mathcal C$ is $(1,1,1,0)$ and the $0^{th}$ Gotzmann coefficient of $\mathcal C$ is 1.

\item[(b)] Now let $\mathcal C$ be a plane cubic curve in $\p^3$.  In this case the Hilbert function of $\mathcal C$ is
$$
1, \ 3, \ 6, \ 9, \ 3z, \ldots
$$
and so the Hilbert polynomial of $\mathcal C$ is $P_{\mathcal C}(z) = 3z$.  It follows that the Gotzmann difference set of $\mathcal C$ is $(1,1,1)$ and so the $1^{th}$ Gotzmann coefficient of $\mathcal C$ is 3 but $C_0(\mathcal C) = 0$.

\end{enumerate}
\end{Ex}

As one can see from these considerations about maximal growth, it is important to know when we can be sure that maximal growth persists.  In \cite{Preser} Preser made the following definition:

\begin{Def}\label{perindex} Let $I$ be a homogeneous ideal of $R$ and let $A = R/I$.  The {\it persistence index} of $A$ is the least integer $d$ such that the Hilbert function of $A$ has maximal growth in all degrees $\geq d$.
\end{Def}

There are some very interesting characterizations of the persistence index, which we summarize in the following theorem.

\begin{Thm} \label{persistence number} Let $I \subset R$ be a homogeneous ideal and let $A = R/I$.  Then
\begin{enumerate}
  \item [a)] the persistence index of $A$ is the maximum degree of a minimal generator for the lex segment ideal $J$ with the property that $B = R/J$ has the same Hilbert function as $A$.
  \item [b)] if $I$ is a saturated ideal, the persistence index of $A$ is $G(A)$, the Gotzmann number of $A$.
  \item [c)] if $\X \subset \p^n$ is a closed subscheme and $H_\X$ its Hilbert function, then the least integer $d$ for which
$$
 H_\X(j) ^{\mac{j}} =  H_\X(j+1) \hbox{ for all } j \geq d
$$
is $d = G(\X)$, the Gotzmann number of $\X$.
\end{enumerate}
\end{Thm}
\medskip

Observe that $c)$ is an immediate consequence of $b)$ and both $a)$ and $b)$ are proved in \cite{AC} and \cite{KR}. 

\medskip 

We saw earlier that if there is maximal growth in the Hilbert function of $A$ for all $t \geq d$ then the Macaulay difference set of $H(A,d)$ determines the Hilbert polynomial, $P_A(z)$, and, moreover, $H(A,t) = P_A(t)$ for all $t \geq d$.  This last condition is strongly connected with the notion of {\it Castelnuovo-Mumford regularity}, which we now recall.

\vskip .5cm
Let $M$ be a finitely generated graded $R$-module and let
$$
0\rightarrow E_n \rightarrow \cdots\rightarrow E_1\rightarrow E_0 \rightarrow M \rightarrow 0
$$
be a minimal graded free resolution of $M$, where
$$
E_p = \oplus_{j}R(-j)^{\beta_{p,j}} .
$$
We call $\beta_{p,j}$ the $p^{th}$ {\it Betti} number of degree $j$.  We say that $M$ is {\it $\ell$-regular} if, whenever $\beta_{p,j} \neq 0$ we have $j-p \leq \ell$.  The {\it Castelnuovo-Mumford regularity of $M$} (or simply the {\it regularity of $M$}) is the least integer $\ell$ so that $M$ is $\ell$-regular.  We will write $\reg(M) = \ell$.

One of the more useful (for us) properties of the regularity of a saturated ideal $I$ in the polynomial ring $R$ is the following:

\begin{Thm}\label{heqp} Let $I$ be a saturated ideal in the polynomial ring $R$. If $H(R/I,-)$ and $P(R/I,-)$ are the Hilbert function and polynomial, respectively, of $R/I$ then
\begin{equation}
H(R/I,d) = P(R/I,d) \hbox{ for all } d \geq \reg(I) - 1.
\end{equation}
\end{Thm}

There is a wonderful theorem of G. Gotzmann which relates the regularity of the defining ideal of a closed subscheme $\X \subset \p^n$ to what we have discussed above.  More precisely

\begin{Thm}[\sc Gotzmann's Regularity Theorem]\label{regularity theorem}
Let $\X$ be a closed subscheme of $\p^n$ with defining ideal $I_{\X}$. If $G(\X)$ is the Gotzmann number of $\X$ then $I_\X$ is $G(\X)$-regular.
\end{Thm}

\begin{Remk} Although we obtain equality between the Hilbert function and the Hilbert polynomial for all $d \geq \reg(I) - 1$, this does not force the Gotzmann number to be $\leq \reg(I) - 1$.  Indeed, for the rational normal curve in $\p^3$ we see that its defining ideal has regularity 1 but the Gotzmann number of the rational normal curve in $\p^3$ is 4.
\end{Remk}

\vskip .5cm
We also need to recall how the Gotzmann number and $i^{th}$ Gotzmann coefficients change when we pass from a variety $\X \subset \p^n$ to its general hyperplane section in $\p^{n-1}$.

\begin{Remk} Let $\X$ be a closed subscheme of $\p^n$ defined by the ideal $I = I_\X$ and let $\mathcal H$ be a general hyperplane of $\p^n$ defined by the general linear form $L$.  Since $I_\X$ is a saturated ideal, multiplication by $\overline{L}$ is an injective linear transformation on the homogeneous pieces of the coordinate ring of $\X$.  It follows that, for large degrees $d$,
$$
P_{\X \cap {\mathcal H} }(d) = P_\X(d) - P_\X(d-1):=\Delta P_\X(d).
$$
On the other hand, for large degrees the growth of the Hilbert function is the maximum permitted by Macaulay's theorem.  Thus by Theorem \ref{T:202} (b), (and noting that $\delta > 1$ for $d$ large enough - see Remark~\ref{delta>1}) we obtain that
\begin{equation}\label{EQ:202}
P_{\X\cap{\mathcal H} }(d) = P_{\X}(d)_{\mac{d}} .
\end{equation}
It follows that
\begin{equation}\label{hyp section}
P_{\X\cap{\mathcal H}}(d) = \Delta P_\X(d) = P_\X(d)_{\mac{d}}, \hbox{ for all $d \gg 0.$ }
\end{equation}
\end{Remk}

We obtain the following easy consequences of these observations (see e.g., \cite{AC}).

\begin{Thm}\label{T:204} Let $\X \subset \p^n$ be a closed subscheme of dimension $r > 0$ and $\mathcal H$ a general hyperplane of $\X$ then
\begin{enumerate}
  \item [(a)]
$$
C_i(\X) = C_{i-1}(\X \cap {\mathcal H})
$$
as long as $i-1 \geq 0$;
\item[(b)]
$$
G(\X) - G(\X \cap {\mathcal H}) = C_0(\X) .
$$
\end{enumerate}
\end{Thm}

\begin{Remk} \label{R:205}
Of course one can continue this line of argument for successive hyperplane sections. One no longer necessarily has that the ideal under consideration is saturated, but that is not really important since our interest is only in the multiplication map by a general linear form {\it in high degrees}. That multiplication, in high degrees, is injective as long as the ideal we are considering does not have radical equal to the irrelevant ideal of $R$. That was really the only thing we used in the discussion above. So, let $\Lambda_i$ be a general linear variety of dimension $n-i$.  We define
$$
G_i(\X) := G(\X\cap \Lambda_i) .
$$
It follows from our remarks above, that
$$
C_i(\X) = C_0(\X\cap \Lambda_i)
$$
and, if we set $G_{r+1}(\X) = 0$ (recall $r = \dim \X$) we obtain, for $i=0,1,...,r$,
$$
C_i(\X)=G_i(\X)-G_{i+1}(\X).
$$

It follows that
\begin{equation}\label{sht}
\deg (\X)  \le G_r(\X) \le G_{r-1}(\X) \le \cdots \le G_0(\X)= G(\X).
\end{equation}

Moreover, by equation~\eqref{degree} and the remark above, we have
$$
\deg(\X)=G_r(\X)=C_r(\X).
$$
\end{Remk}

In the case of varieties of low dimension we can reinterpret some of these results in terms of things already known and defined.  For example, it is easy to prove the following result.

\begin{Thm}[Theorem 4.7 \cite{AC}]\label{T:205}
Let $\X$ be a closed subscheme in $\p^n$ of dimension $r$
and let $p_a(\X)$ be the arithmetic genus of $\X$. If we let $G_i$ denote
$G_i(\X)$  for $0\le i\le r+1$. Then
\begin{enumerate}
\item[(a)] If $\X$ is a zero dimensional scheme in $\p^n$,
$G(\X)=C_0(\X)=\deg (\X)$.
\item[(b)] If $\X$ is a projective curve in
$\p^n$ then,
\[C_0(\X)=\binom{\deg (\X)-1}{2}-p_a(\X).\]
\item[(c)] If $r=\dim (\X) \geq 2$ then
\[
C_0(\X)=(-1)^{r+1}\left[\binom{G_r-1}{r+1}-p_a(\X)\right]-\sum_{s=1}^{r-1}
(-1)^s\left[\binom{G_s-1}{s+1}-\binom{G_{s+1}-1}{s+1}\right] .\]
\end{enumerate}
\end{Thm}

We will finish this section of preliminaries by recalling a property that certain sets of points in $\p^n$ might enjoy.

\begin{Def}\label{UPP} Let $\X$ be a set of $s$ distinct points in $\p^n$.  We say that $\X$ has the {\it uniform position property} (abbreviated {\it UPP}) if for every integer $d \leq s$, every $d$ subset of $\X$ has the same Hilbert function.
\end{Def}

The importance of this notion comes from the fact that if $\X$ is a non-degenerate irreducible closed subvariety of dimension $d$ and degree $s$ in $\p^n(k)$ ($k$ algebraically closed of characteristic 0) then the points obtained from $\X$ by successively cutting $\X$ with $d$ general hyperplanes gives us a set of $s$ points in $\p^{n-d}$ with UPP (Eisenbud and Harris \cite{EH1}).  It is a long outstanding problem to characterize the Hilbert functions of set of points with UPP.  A complete characterization is only known for points in $\p^2$ (see Maggioni-Ragusa~\cite{MR}, Geramita-Migliore~\cite{GM-1}).

\section{Extremal Behavior in the Hyperplane Restriction
Theorem}\label{sec_3}

In this section we define a new numerical invariant, $M(\X)$ (derived from Green's Hyperplane Restriction Theorem (Theorem \ref{T:201} (b))) for any closed subscheme $\X$ of $\p^n$.  Lemma \ref{L:303} is the key to the main results of this section.   Using it we can prove Theorem~\ref{T:309}, a slight generalization of Theorem 4 in \cite{G1}, and Theorem \ref{T:311}.  The latter gives a necessary and sufficient condition for a scheme $\X$ to satisfy $G(\X)=\deg (\X)$.

Let $\X$ be a closed subscheme in $\p^n$ with homogeneous
saturated ideal $I_\X$ and set $S_\X:=R/I_\X$. Then, by Theorem~\ref{T:201},
\begin{equation}\label{meq}
H(S_\X/hS_\X,d) \le H(S_\X,d)_{\langle d\rangle },
\end{equation}
for a general linear form $h$ in $R$ and for all $d\geq 1$. By analogy with the notion of the Persistence Index, we define the numerical invariant $M(\X)$ to be the minimum degree
where the equality begins to persist in equation (\ref{meq}), that
is
$$
M(\X)=\min \{d \,|\, H(S_\X/hS_\X,t)=H(S_\X,t)_{\langle t\rangle }\, \textup{for
all}\, t\geq d \}.
$$

\vskip .5cm Since $I_\X$ is a saturated ideal, we have the following equality for any general linear form $h$:
\begin{equation*}
\Delta H_\X(d)=H(S_\X/hS_\X,d), \qquad \forall d\ge 1.
\end{equation*}
 Hence, by equation~\eqref{meq},
\begin{equation}\label{deq}
\begin{array}{clllllllllllll}
\Delta H_\X(d) =H(S_\X/hS_\X,d)\le H(S_\X,d)_{\langle d\rangle}=H_\X(d)_{\langle d\rangle }, \quad \text{i.e.,} \\[1ex]
\Delta H_\X(d) \le H_\X(d)_{\langle d\rangle }
\end{array}
 \end{equation}
 for such $d$, and thus, by equation~\eqref{deq}, we can rewrite $M(\X)$ as
\begin{equation}\label{EQ:303}
\begin{array}{lllllllllllllllll}
M(\X)
& = & \min \{d \,|\, H(S_\X/hS_\X,t)=H(S_\X,t)_{\langle t\rangle },\, \forall t\geq d \} \\[.5ex]
& = & \min \{d \,|\, \Delta H_\X(t)=H_\X(t)_{\langle t\rangle }, \, \forall t\geq d \}.
\end{array}
\end{equation}

Recall that by Theorem \ref{persistence number}
$$
G(\X)=\min\{d \mid H_\X(t+1)=H_\X(t)^{\langle t\rangle},\, \forall t\ge d\}.
$$
So, applying Theorem~\ref{T:202} (b), for every $t \ge G(\X)$, we get
$$
\begin{array}{cccllllllllllllll}
H(S_\X/hS_\X,t+1) & = & H(R/I_\X,t+1)_{\langle t+1\rangle} \\
\Vert &   & \Vert \\
\Delta
H_\X(t+1)& = & H_\X(t+1)_{\langle t+1\rangle },
\end{array}
$$

In other words, $M(\X)\le G(\X)+1$. We can, however, improve this inequality.

\begin{Prop} \label{P:301}
Let $\X$ be a closed subscheme in $\p^n$. Then
$$
M(\X)\leq G(\X).
$$
\end{Prop}

\begin{proof} It suffices to show, by equation~\eqref{EQ:303}, that for $g=G(\X)$
$$
\Delta H_\X(g) = H_\X(g)_{\mac{g}}.
$$
\smallskip

Let $(a_s, a_{s-1}, \ldots , a_\delta)$ be the Gotzmann difference set of the Hilbert polynomial $P_\X$.  Then
$$
P_\X(d)=\binom{a_s+d}{d} + \binom {a_{s-1}+(d-1)}{d-1}+\cdots +
\binom{a_\delta+(d-s+1)}{d-s+1}
$$
Then, by Theorem \ref{persistence number} (see also Theorem 2.6 (ii) in \cite{AC}), we have $s- \delta +1=g$, and thus  $I_\X$ is $g$-regular by Gotzmann's Regularity Theorem \ref{regularity theorem}.

\smallskip

Note that $H_\X(d)=P_\X(d)$ for any $d\geq g-1$ (Theorem \ref{heqp}, or see also Lemma 2.5 (iii) in \cite{AC}). Hence we have
$$
\begin{array}{lllllllllllllllll}
\Delta H_\X(g) & = & \Delta P_\X(g), \quad  \text{and} \\[2ex]
H_\X(g)_{\langle g\rangle } & = & P_\X(g)_{\langle g\rangle }.
\end{array}
$$
Now, it is enough to show that
$$
\Delta P_\X(g)=P_\X(g)_{\langle g\rangle}.
$$

Recall Pascal's identity $\binom{a_i+i-1}{i-1}+\binom{a_i+i-1}i=\binom{a_i+i}i$.

Let $C_0(\X)=c_0\ge 0$ and consider
$$
P_\X(z) = \binom{a_s + z}{a_s} + \cdots + \binom{a_{c_0 + \delta} + z - (s-\delta-c_0)}{a_{c_0 + \delta}} + c_0 .
$$
Since $s-\delta+1=g$, we see that
\begin{align*}
P_\X(g) =&\binom{a_s + g}{a_s} + \cdots + \binom{a_{c_0 + \delta} + g - (s-\delta-c_0)}{a_{c_0 + \delta}} + c_0\\[4ex]
=&\binom{a_s + g}{a_s} + \cdots + \binom{a_{c_0 + \delta} + c_0 + 1}{a_{c_0 + \delta}} + c_0\text{ and }\\[4ex]
P_\X(g-1)=&\binom{a_s+g-1}{a_s}+\cdots +\binom{a_{c_0+\delta}+ g-1 - (s-\delta -c_0)}{a_{c_0 +\delta}}+c_0
\\[4ex]
=&\binom{a_s+g-1}{a_s}+\cdots +\binom{a_{c_0+\delta}+ c_0}{a_{c_0 +\delta}}+c_0\\
\end{align*}
where $a_{c_0+\delta}\ge  1$. Applying Pascal's identity again,
$$
\begin{array}{llllllllllllll}
\Delta P_\X(g) & = & \displaystyle \binom {a_s+g-1}{a_{s}-1} +\cdots +
\binom{a_{c_0+\delta}+c_0}{a_{c_0+\delta}-1}\\[4ex]
& = & P_\X(g)_{\langle g\rangle }.
\end{array}
$$
as we wanted.
\end{proof}

\medskip\medskip

\begin{Lem}\label{L:302} Let $\X$ be as above. Then
$$
\Delta H_\X(d)-(\Delta H_\X(d))_{\langle d\rangle } \le \Delta H_\X(d-1)\le
H_\X(d-1)_{\langle d-1\rangle }
$$
for every integer $d \geq 1$.
\end{Lem}

\begin{proof} Since the second inequality is always true (see Theorem~\ref{T:201} b)), it is enough to verify the first inequality.

Let $J = (L_1,L_2)$ be the ideal generated by any two general linear forms in $R=k[x_1,\dots,x_n]$ and let $K = I_\X + (L_1)$. Then, multiplication by $L_2$ gives the exact sequence
\begin{equation*}
0 \rightarrow  ( (K:L_2) /K )_{d-1} \rightarrow
(R/K)_{d-1} \stackrel{\times L_2}{\longrightarrow} (R/K)_{d}
\rightarrow (R/(I_\X + J))_{d} \rightarrow 0.
\end{equation*}
Since $R/K$ is the coordinate ring of the general hyperplane section of $\X$ and its Hilbert function is $\Delta H(\X,-)$, we have (by taking the alternating sum of the dimensions of the graded pieces of this exact sequence)
$$
\begin{array}{lllllllllllllllllllll}
\Delta H_\X(d)-\Delta H_\X(d-1)
& = & H(R/K,d) - H(R/K,d-1)\\[.5ex]
& = & H(R/(I_\X + J),d)-\dim_k  ((K:L_2)/K)_{d-1}\\[.5ex]
& \le  & H(R/(I_\X + J),d)\\[.5ex]
& = & H(R/(K+(L_2)),d)\\[.5ex]
& \le & H(R/K,d)_{\langle d\rangle} \\[.5ex]
& = & (\Delta H_\X(d)) _{\langle d\rangle},
\end{array}
$$
i.e.,
$$
\Delta H_\X(d)-(\Delta H_\X(d))_{\langle d\rangle } \le \Delta H_\X(d-1),
$$
which proves the first inequality and thus finishes the proof of the Lemma.
\end{proof}

\medskip\medskip

The following lemma is the pivotal result of this section.

\begin{Lem}\label{L:303}
Suppose that $\Delta H_\X(d)= H_\X(d)_{\mac d}$ and $C_0(\X,d)=0$.
Then,
$$
\Delta H_\X(d-1) = H_\X(d-1)_{\mac{d-1}} \hbox{ and } C_{0}(\X,d-1)=0.
$$
\end{Lem}

\begin{proof} Let $I_\X \subset S = k[x_0, \ldots , x_n]$ and write $S_\X = S/I_\X$.
We first write the $d^{th}$ binomial expansion of $H_\X(d)$:
$$
H_\X(d)=\binom {a_d+d}d+ \binom
{a_{d-1}+(d-1)}{d-1}+ \cdots + \binom {a_\delta +\delta}\delta
$$
with $a_d\geq a_{d-1} \geq \cdots \geq a_\delta \geq 1$ and $\delta \geq 1$. Note that $a_\delta \geq 1$ since we are assuming that $C_0(\X,d) = 0$.

\smallskip

First we will prove that $C_0(\X,d-1)=0$. We divide this argument into two cases: $\delta >1$ and $\delta = 1$.

Now assume $\delta >1$. With the notation as above and with the hypothesis that $\Delta H_\X(d)=H_\X(d)_{\langle d\rangle}$ (which we can rewrite as $H(S_\X/hS_\X,d)=H(S_\X,d)_{\langle d\rangle}$)   Theorem~\ref{T:202} gives us
$$
H_\X(d)=H_\X(d-1)^{\langle d-1\rangle },
$$
i.e.,
\begin{equation}\label{eq3-1}
H_\X(d-1)=\binom{a_{d}+(d-1)}{d-1}+ \cdots +
\binom{a_{\delta}+(\delta-1)}{\delta-1}.
\end{equation}
Moreover, since  $a_\delta \geq 1$, we have $C_0(\X,d-1)=0$ in equation (\ref{eq3-1}), as we wanted to prove.

\smallskip

Now assume $\delta=1$. Then, by Corollary \ref{C:203} we have
$$
\begin{array}{llllllllllllllll}
&   & H_\X(d-1)^{\mac {d-1}} \\[.5ex]
& = &  H_\X(d)+a_2-a_1+1\\[1ex]
& = & \displaystyle \binom {a_d+d}d+ \binom
{a_{d-1}+(d-1)}{d-1}+ \cdots + \binom{a_{2}+2}{2}+\binom {a_1 +1}{1} +a_2-a_1+1\\[2.5ex]
& = & \displaystyle \binom{a_{d}+d}{d}+ \binom
{a_{d-1}+(d-1)}{d-1}+  \cdots + \binom{a_{2}+2}{2}+\binom{a_2+2}{1}.
\end{array}
$$
Note that since we are assuming that $C_0(\X,d)=0$,we get that $a_2=\alpha\geq 1$.  Let $\beta=\max\{ \ell\, |\,a_\ell=a_2=\alpha\}.$ Then $\beta\geq 2$ and
$$
\begin{array}{llllllllllllllll}
  & H_\X(d-1)^{\mac{d-1}}\\[1ex]
= & \displaystyle \binom{a_d+d}d+\cdots
+\binom{a_{\beta+1}+\beta+1}{\beta+1}+\binom{\alpha+\beta}{\beta}+\cdots
+\binom{\alpha+2}{2}+\binom{\alpha+2}{1}
\end{array}
$$
which, by repeated use of the identities in Pascal's Triangle, gives
$$
\begin{array}{llllllllllllllll}
= & \displaystyle \binom{a_d+d}d+\cdots
+\binom{a_{\beta+1}+\beta+1}{\beta+1}+\binom{(\alpha+1)+\beta}{\beta}.
\end{array}
$$

This last equation is precisely the $d$-binomial expansion of $H_\X(d-1)^{\mac{d-1}}$, and so we deduce that
\begin{equation}\label{EQ:306-1}
H_\X(d-1)=\binom{a_{d}+d-1}{d-1}+ \cdots
+\binom{a_{\beta+1}+\beta}{\beta}+ \binom{\alpha+1+(\beta-1)}{(\beta-1)},
\end{equation}
and hence $C_0(\X,d-1)=0$.

\smallskip

That completes the proof that $C_0(\X,d-1)=0$ and so all that remains to prove is that
\begin{equation}\label{EQ:306A}
\Delta H_\X(d-1) = H_\X(d-1)_{\mac{d-1}} .
\end{equation}

We already have, from \ref{L:302} that
$$
\Delta H_\X(d-1) \leq H_\X(d-1)_{\mac{d-1}}
$$
so it remains only to show the reverse inequality.

Suppose we could show that
\begin{equation} \label{EQ:306}
H_\X(d-1)_{\langle d-1\rangle }=H_\X(d)_{\langle d\rangle }-(H_\X(d)_{\langle d\rangle })_{\langle d\rangle }
\end{equation}
 then
$$
\begin{array}{llllllllllllll}
H_\X(d-1)_{\langle d-1\rangle } &  = & H_\X(d)_{\langle d\rangle }-(H_\X(d)_{\langle d\rangle })_{\langle d\rangle }\\[.5ex]
& = & \Delta H_\X(d)-(\Delta H_\X(d))_{\langle d\rangle } & (\text{since } \Delta H_\X(d)=H_\X(d)_{\langle d\rangle }) \\[.5ex]
& \le & \Delta H_\X(d-1) & ( \text{by Lemma~\ref{L:302}}) .
\end{array}
$$
That would establish equation~\eqref{EQ:306A} and we would be done.

\smallskip

It remains to show that equation~\eqref{EQ:306} is true.

\smallskip

From equations~\eqref{eq3-1} (covering the case when $\delta > 1)$ and~\eqref{EQ:306-1} (covering the case when $\delta =1$), we have that
\begin{equation} \label{EQ:309}
\begin{array}{llllllllllllllllllll}
& H_\X(d-1)_{\mac {d-1}}\\[1ex]
= & \begin{cases} \displaystyle {\binom{a_{d}+d-2}{d-1}+ \cdots + \binom{a_{\beta+1}+\beta-1}{\beta}+\binom{\alpha+\beta-1}{\beta-1}}, & \text {if } \delta=1 \\[2ex]
\displaystyle {\binom{a_{d}+d-2}{d-1}+ \cdots +\binom{a_{\delta}+\delta
-2}{\delta-1}}, & \text {if } \delta >1.
\end{cases}
\end{array}
\end{equation}

Now, we compute
\begin{equation*}
H_\X(d)_{\mac d}-\left(H_\X(d)_{\mac d}\right)_{\mac d}.
\end{equation*}

Assume $\delta=1$ and let $a_2=\alpha \ge 1$ and $\beta=\max \{\ell |a_\ell =a_2=\alpha\}$ as above. Note that $\binom{a_1}{1}-\binom{a_1-1}{1}=1$ for $a_1\geq 1$.  Then we have

\begin{equation} \label{EQ:310}
\begin{array}{lllllllllllllllllll}
  & H_\X(d)_{\mac d}-\left(H_\X(d)_{\mac d}\right)_{\mac d}\\[2ex]
= & \displaystyle \left[\binom{a_d+d-1}d+\cdots +
     \binom{a_{\beta+1}+\beta}{\beta+1}+
\binom{\alpha+\beta-1}{\beta}+\cdots+\binom{\alpha+1}2+\binom{a_1}1\right] - \\[3ex]
  &   \displaystyle \left[\binom{a_d+d-2}d+\cdots +
     \binom{a_{\beta+1}+\beta-1}{\beta+1}+
\binom{\alpha+\beta-2}{\beta}+\cdots+\binom{\alpha}2+\binom{a_1-1}1\right] \\[3ex]
= & \displaystyle \binom{a_d+d-2}{d-1}+ \cdots+\binom {a_{\beta+1}+\beta-1}{\beta}+\binom{\alpha+\beta-2}{\beta-1}+\cdots +\binom {\alpha}{1} +1\\[3ex]
= & \displaystyle \binom{a_{d}+d-2}{d-1}+ \cdots
+\binom{a_{\beta+1}+\beta-1}{\beta}+\binom{\alpha+\beta-1}{\beta-1}.
\end{array}
\end{equation}
This last (by equation~\eqref{EQ:309} for $\delta = 1$) is $= H_\X(d-1)_{\mac {d-1}}$, and we are done in this case.

Now assume $\delta >1$. Then we have
\begin{equation} \label{EQ:311}
\begin{array}{llllllllllllllllll}
   & H_\X(d)_{\mac d}-\left(H_\X(d)_{\mac d}\right)_{\mac d}\\[1ex]
 = & \displaystyle\left[ \binom{a_d+d-1}d+\cdots +\binom{a_\delta +\delta -1} \delta  \right]
 -\left[\binom{a_d +d-2}d + \cdots + \binom{a_\delta +\delta -2} \delta \right]\\[2ex]
 = & \displaystyle \binom{a_d+d-2}{d-1}+\cdots+\binom{a_{\delta}+\delta-2}{\delta-1}.
\end{array}
\end{equation}
This last (by equation~\eqref{EQ:306} for $\delta > 1$) $= H_\X(d-1)_{\mac {d-1}}$.

This finishes the proof that equation~\eqref{EQ:306} is always true and finishes the proof of the lemma.
\end{proof}

\begin{Remk}\label{rem}
 The Lemmata above establish the following strange sounding result: suppose there exists a positive integer $d$ such that $\Delta H_\X(d)=H_\X(d)_{\langle d\rangle }$ {\bf and} $C_0(\X,d)=0$. Then we
have
$$
\Delta H_\X(\ell)=H_\X(\ell)_{\langle \ell\rangle }, \quad \text{and} \quad  C_0(\X,\ell)=0
\quad \text{for any} \quad  \ell \le d.
$$
I.e., for every $\ell \leq d$,  $H(\X,\ell)$ is determined by $H(\X,d)$ using Theorem~\ref{T:202} and Corollary~\ref{C:203}.  We will use this several times in the sequel.
\end{Remk}

\medskip\medskip

Our first use for Lemma~\ref{L:303} comes out of a reflection about Proposition~\ref{P:301}, which asserted that $M(\X)\leq G(\X)$. This inequality raises the following natural question:

\begin{ques}
What does   $M(\X)\neq G(\X)$ mean?  I.e., what is the significance of $M(\X)< G(\X)$?
\end{ques}

Using Lemma \ref{L:303}, we obtain the following surprising result.

\begin{Prop}\label{P:306} Let $\X$ be a closed subscheme of $\p^n$: if $M(\X) < G(\X)$ then $M(\X) = 1$.
\end{Prop}

\begin{proof}
Suppose that  $M(\X)<G(\X)$ and write $G(\X)=t+1$.

\smallskip

Recall that
$$
\begin{array}{llllllllllllllllllllllll}
G(\X) & = & \min\{ d  \mid H_\X(n+1)=H_\X(n)^{\langle n \rangle}, \forall n\ge d\},
\quad \text{and} \\[1.5ex]
M(\X) & = & \min\{ d \mid \Delta H_\X(n)=H_\X(n)_{\langle n\rangle}, \forall n\ge d\}.
\end{array}
$$

In view of Theorem \ref{persistence number} (see also Lemma 2.5 in \cite{AC}) we obtain
$$
H_\X(t+1)  =  P_\X(t+1) .
$$
Since $M(\X) < G(\X)$ it follows that
$$
\Delta H_\X(t)  =  H_\X(t)_{\langle t\rangle }.
$$
If we knew that $C_0(\X,t)=0$ we could apply Lemma~\ref{L:303} and Remark~\ref{rem} to conclude that $M(\X) = 1$.  We now seek to show exactly that.

Now, it also follows from $M(\X) < G(\X)$ that
$$
\Delta H_\X(t+1)  =  H_\X(t+1)_{\langle t+1\rangle }.
$$

So, let the $(t+1)$-st binomial expansion of $H_\X(t+1)$ have the form:
\begin{equation} \label{EQ:312}
\begin{array}{lllllllllllllllll}
H_\X(t+1)
& = & P_\X(t+1) \\[.5ex]
& = & \displaystyle
\binom{a_{t+1}+t+1}{t+1}+ \cdots +\binom{a_{\delta}+\delta}{\delta}.
\end{array}
\end{equation}

If $\delta>1$, then, by Theorem 3.3 in \cite{AC}, we have
$$
H_\X(t+1)=H_\X(t)^{\langle t\rangle},
$$
i.e., $G(\X)\le t$, a contradiction.

Hence we must have  $\delta=1$, and so we can rewrite equation~\eqref{EQ:312} as
$$
\begin{array}{lllllllllllllllll}
H_\X(t+1)
& = & P_\X(t+1) \\[.5ex]
& = & \displaystyle
\binom{a_{t+1}+t+1}{t+1}+ \cdots +\binom{a_{1}+1}{1}.
\end{array}
$$

Let $a_2=\alpha \ge 0$ and $\beta=\max\{\ell \mid a_\ell =a_2=\alpha \}$.
Then, by Corollary \ref{C:203}
$$
\begin{array}{lllllllllllllllll}
H_\X(t)^{\mac t}
& = &  H_\X(t+1)+a_2-a_1+1\\[.5ex]
& = &  \displaystyle \binom{a_t+t+1}{t+1} + \cdots +\binom{a_{\beta + 1}+\beta + 1}{\beta + 1} \\ [2ex]
&   &  \displaystyle + \binom{\alpha + \beta}{\beta} + \cdots +\binom{\alpha+2}2+\binom{a_1+1}1+\alpha-a_1+1\\[2ex]
& = &  \displaystyle \binom{a_t+t+1}{t+1} + \cdots +\binom{a_{\beta + 1}+\beta + 1}{\beta + 1} \\ [2ex]
&   &  \displaystyle  \binom{\alpha + \beta}{\beta} + \cdots +\binom{\alpha+2}2+\binom{\alpha +2}1 \\[2ex]
& = &  \displaystyle \binom{a_t+t+1}{t+1}+ \cdots
       +\binom{a_{\beta+1}+\beta+1}{\beta+1}+\binom{(\alpha+1)+\beta}{\beta}.
\end{array}
$$
Hence,
\begin{equation} \label{EQ:313}
\begin{array}{lllllllllllllll}
  H_\X(t)
& = & \displaystyle \binom{a_{t+1}+t}t + \cdots + +\binom{a_{\beta+1}+\beta }{\beta }+\binom {(\alpha+1)+(\beta-1)}{(\beta-1)} . \\
\end{array}
\end{equation}
Since $\alpha + 1 > 0$ we get that $C_0(\X,t) = 0$ and hence
$M(\X)=1$, as we wished.
\end{proof}

\medskip\medskip  From Proposition~\ref{P:306} one sees that it is important to understand the situation when $M(\X) = 1$.  In case $\X$ is a reduced and equidimensional subscheme of $\p^n$  we will obtain a complete description for when $M(\X) = 1$ occurs (see Corollary~\ref{M(X)=1}).

What is required for that characterization is an understanding of when equality occurs between two other invariants of $\X$, namely $G(\X)$ and $\deg(\X)$.  By definition (see also equation \ref{degree}) we always have $G(\X) \geq \deg(\X)$, so a natural question is: when is $G(\X) = \deg(\X)$?

The following examples are instructive.

\begin{Ex}\label{Ex307}

$a)$ Let $\X$ be a finite set of points in $\p^n$ such that $|\X|=d>1$. It can be easily verified, using Lemma~\ref{L:303} and Proposition~\ref{P:306}, that the following  are equivalent:
\begin{enumerate}
\item[(a)] the points of $\X$ are colinear.
\item[(b)] $d = G(\X)$ and $ 1 = M(X)$.
\end{enumerate}

In fact, if the points of $\X$ are colinear and $|\X|=d$, then we have
$$
H_\X(t)=
\begin{cases}
\displaystyle \binom{t+1}{t} = t+1 , &   \text{for }  t=0,1,\dots,d-1, \\[1.5ex]
d, & \text{for } t\ge d,
\end{cases}
$$
i.e.,
$$
\begin{array}{lllllllllllllll}
G(\X) =  d
 \ne  1
 =  M(\X) \quad
\end{array}
$$
(in view of  Lemma~\ref{L:303}, Remark~\ref{rem} and Proposition~\ref{P:306}).

Conversely, assume $G(\X)\ne M(\X)$.  Then $M(\X)=1$ and since $|\X|=d$, we have that
$$
H_\X(t)=d, \quad \forall t\ge d-1.
$$
Moreover, since $H_\X(d-1)_{\langle d-1\rangle} =\binom{d}{d-1}_{\langle d-1\rangle}= 1$, we have that
$$
\begin{array}{lllllllllllllllll}
1
& = & H_\X(d-1)_{\langle d-1\rangle} \\
& = & \Delta H_\X(d-1) \qquad (\hbox{since } M(\X)=1)\\
& = & H_\X(d-1)-H_\X(d-2) \\
& = & d -H_\X(d-2),
\end{array}
$$
and so
$$
H_\X(d-2)=d-1.
$$
By continuing this process with Lemma~\ref{L:303} and Remark~\ref{rem}, we see that
$$
H_\X(t)=t+1, \quad \forall t\le d-1,
$$
which means that the points of  $\X$ are collinear, as we desired.

\smallskip

$b)$ Let $\X$ be a closed subscheme in $\p^n$ where $\dim \X=r$. Recall that
$$
G_0(\X):=G(\X) \quad \text{and} \quad G_i(\X):=G(\X\cap \Lambda_i)
$$
where $\Lambda_i$ is a general linear subspace of dimension $r-i$ for $i=1,\dots, r$.

For a curve $\X$ in $\p^n$, it is well known that $\X$ is a {\it plane} curve of degree $d$ if and only if the arithmetic genus of $\X$, $p_a(\X)$, is $p_a(\X)=\binom{d-1}{2}$. Indeed, note that, by Theorem \ref{T:205},
$$
C_0(\X)=\binom{d-1}{2}-p_a(\X).
$$

If we assume $C_0(\X)=0$ we conclude that $\X$ is a plane curve, and so
$$
G(\X)=C_0(\X)+C_1(\X)=C_1(\X)=\deg (\X) \quad (\hbox{because } C_0(\X)=0).
$$
It follows that
$$
\begin{array}{llllllllllll}
 C_0(\X)=0 & \Leftrightarrow
& G(\X)=\deg(\X) \\
&\Leftrightarrow
& \X \text{ is a plane curve.}
\end{array}
$$
\end{Ex}

\medskip\medskip

\begin{Remk}\label{degenerate}
Notice that $d$ points on a line of $\p^n$ ($n\geq 2$) and a plane curve in $\p^n$ ($n \geq 3$) are special degenerate varieties defined by an appropriate collection of linearly independent linear forms and a form of degree $d$, i.e., $I_\X $ is a very special kind of complete intersection variety, namely {\it a hypersurface in a proper linear subspace of $\p^n$}.  Our next goal is to show that the examples above are indicative.  Such varieties are characterized by the equality $G(\X) = \deg(\X)$.
\end{Remk}

\medskip\medskip
We begin our investigation of this equality between $G(\X)$ and $\deg(\X)$ by recalling a lovely result of Green in \cite{G1}.

\begin{Prop}[Theorem~3 \cite{G1}] \label{P:307}
Let $\X$ be a closed subscheme of $\p^n$.

$i)$ If
$$
H_\X(d)= \binom{r+d}{d}\quad {and}\quad \Delta H_\X(d) =
\binom{(r-1)+d}{d}
$$
for some $d\geq 1$ and $r\geq 1$, then
$$
(I_\X)_d=(I_{\Lambda})_d
$$
for some $r$-dimensional linear space $\Lambda$ in $\p^n$;

\smallskip

$ii)$  if there are integers $d \geq \ell \geq 0$ such that
$$
H_\X(d)=\, \binom{1+d}{d}+\cdots+\binom{1+(d-\ell+1)}{(d-\ell+1)}
$$
{and}
$$
 \ell = \Delta H_\X(d) = H_\X(d)_{\mac d}
$$
  then
$$
(I_\X)_d=(I_C)_d,
$$
where $C$ is a plane curve of degree $\ell$.
\end{Prop}

We now generalize Proposition~\ref{P:307} using Lemma~\ref{L:303}.

\begin{Thm}\label{T:309} Let $\X$ be a closed subscheme of $\p^n$.
If there are integers $d,\ \ell, \ r$, where $1 \leq \ell \leq d$ and $r \geq 1$ for which
$$
H_\X(d)=\binom{r+d}{d}+\cdots+\binom{r+(d-\ell +1)}{(d-\ell +1)}
$$
and
$$
\Delta H_\X(d) = H_\X(d)_{\mac d}
$$
then
$$
(I_{\X})_d=(I_{{F}_{\Lambda}})_d
$$
where $F_{\Lambda}$ is a hypersurface of degree $\ell $ in some
$(r+1)$-dimensional linear subspace $\Lambda$ of $\p^n$. In other words,
there exist linear forms $L_1,\ldots,L_{n-(r+1)}$ and a
homogeneous form $F$ of degree $\ell $ such that
$$
(I_{\X})_d = \big(L_1,\ldots,L_{n-(r+1)},F\big)_d.
$$
\end{Thm}

\begin{proof}
If $r = 1$, this is precisely Proposition~\ref{P:307}.

\smallskip

Now suppose $r \geq 2$. If $d>\ell$, then, by Theorem~\ref{T:202}, we have
$$
H_\X(d)\,=\,H_\X(d-1)^{\mac {d-1}}.
$$

Using the description of the $d$-binomial expansion of $H_\X(d)$ given in the statement of the theorem, we thus deduce that
\begin{align*}
H_\X(d-1)=&\binom{r+(d-1)}{(d-1)} + \cdots + \binom{r+(d-\ell)}{(d-\ell)}
\end{align*}
which we rewrite as
\begin{align*}
        =&\left[\binom{r+(d-1)}{(d-1)} + \cdots + \binom{r+(d-\ell)}{(d-\ell)}+\binom {r+(d-\ell)}{d-\ell -1}\right]-\binom {r+(d-\ell)}{d-\ell -1}\\
        =&\binom{(r+1)+(d-1)}{(d-1)}-\binom {(r+1)+(d-\ell-1)}{(d-\ell -1)}.
\end{align*}
In exactly the same way we find
$$
H_\X(d)=\binom{(r+1)+d}{d}-\binom {(r+1)+(d-\ell)}{(d-\ell)}.
$$

Notice that both $H_\X(d-1)$ and $H_X(d)$ are the values of the Hilbert function of an ideal of the form $I = (F, L_1, \ldots , L_{n-(r+1)})$, (where the forms $F, L_1, \ldots , L_{n-(r+1)}$ are a regular sequence) in degrees $d-1$ and $d$.  Thus, by Corollary 3.2 in \cite{BGM}, we obtain that $(I_\X)_d$ IS the degree $d$ component of the saturated ideal of a hypersurface of degree $\ell $ inside a linear subspace $\Lambda \cong
\p^{r+1}$ of $\p^n$.


\smallskip

Now assume  $d=\ell$, then by Corollary~\ref{C:203} (since $a_2 = a_1 = r$), we get
$$
\begin{array}{llllllllllllllll}
H_\X(d-1)^{\langle d-1\rangle}
& = & H_\X(d)+1 \\[1ex]
& = & \ds \binom{r+d}{d} + \cdots + \binom{r+2}{2}+ \binom{r+1}{1}+1 \\[2ex]
& = & \ds \binom{r+d}{d} + \cdots + \binom{r+2}{2}+ \binom{r+2}{1} \\[2ex]
& = & \ds \binom{(r+1)+d}d,
\end{array}
$$
and thus
\begin{equation}\label{EQ:313-1}
H_\X(d-1)=\binom{(r+1)+(d-1)}{(d-1)}.
\end{equation}
This implies that $C_0(\X,d-1) = 0$.

Moreover, one of our assumptions is that
$$
\Delta H_\X(d)=H_\X(d)_{\langle d\rangle },
$$
and so, by Lemma~\ref{L:303} and equation~\eqref{EQ:313-1}, we have
\begin{equation}\label{EQ:314}
\Delta H_\X(d-1)=H_\X(d-1)_{\langle d-1\rangle }=\binom{r+(d-1)}{(d-1)}.
\end{equation}
Hence, using Proposition~\ref{P:307} with equations~\eqref{EQ:313-1} and \eqref{EQ:314}, we obtain an $(r+1)$-dimensional linear subspace $\Lambda$ of $\p^n$ such that
\begin{equation}\label{EQ:315}
(I_{\X})_{d-1}=(I_{\Lambda})_{d-1}.
\end{equation}

Since the Hilbert function of $\Lambda$ has maximal growth in degree $d-1$, we have that
$$
\begin{array}{lllllllllllllll}
&   & H_\Lambda(d)-H_\X(d) \\[.5ex]
& = & H_\Lambda(d-1)^{\langle d-1\rangle}-H_\X(d) \\[.5ex]
& = & H_\X(d-1)^{\langle d-1\rangle}-H_\X(d) \quad (\hbox{in view of }
\text{equation~\eqref{EQ:315}})\\[.5ex]
& = & \ds \binom{(r+1)+d}{d} -H_\X(d) \hskip 4.1 true mm \quad (\hbox{by }
\text{equation~\eqref{EQ:313-1}}) \\[2ex]
& = & \ds \binom{(r+1)+d}{d}-\left[\binom{r+d}{d} + \cdots + \binom{r+2}{2}+ \binom{r+1}{1}\right] \\[2ex]
& = & 1,
\end{array}
$$
and hence $I_\X$ has one new generator in degree $d=\ell$. This means that there exists a hypersurface $F_{\Lambda}$ of degree $\ell $ in $\Lambda$ such that
$(I_{\X})_d=(I_{F_{\Lambda}})_d$, in the coordinate ring of $\Lambda$, as we wished.
\end{proof}

We now generalize the examples in Example~\ref{Ex307} as promised. Those examples are now seen to be a special case of the following Theorem.

\begin{Thm}\label{T:311}
Let $\X$ be a closed subscheme of $\p^n$ such that
$\dim(\X)=r$ and $\deg(\X)=d$. Then the following are equivalent:

\begin{itemize}
\item[(a)] $G(\X)=\deg (\X)=d;$

\item[(b)] $\X$ is a hypersurface $F_{\Lambda}$ of degree $d$ in some $(r\!+\!1)$-dimensional linear subspace $\Lambda$ of $\p^n$;

\item[(c)] For a general linear space $\Lambda_i$ of dimension $n-i$ in $\p^n$ and for $0\leq i \leq r-1$,
$$
p_a(\X\cap \Lambda_i) ={\binom{d-1}{r-i+1}}.
$$
\end{itemize}
\end{Thm}
\

\begin{proof} (a)$\Rightarrow$(b)
Note that if $G(\X)=\deg(\X)$, then
$$
\begin{array}{lllllllllllllll}
G(\X) & = & C_r(\X)+C_{r-1}(\X)+\cdots+C_1(\X)+C_0(\X) \\
     & = & \deg(\X)+C_{r-1}(\X)+\cdots+C_1(\X)+C_0(\X) \quad ( \text{using Corollary 4.4 of \cite{AC}})\\
& = & \deg (\X) \\
& = & d,
\end{array}
$$
i.e., we have that
$$
C_r(\X)=d \quad \text{and} \quad C_{r-1}(\X)=\cdots=C_1(\X)=C_0(\X)=0.
$$
Hence we see that
\begin{equation}\label{EQ:316}
P_\X(z)=\binom{r+z}{z}+\cdots+\binom{r+(z-d+1)}{(z-d+1)}.
\end{equation}
By Theorem~\ref{T:309}, $(I_{\X})_d=(I_{F_{\Lambda}})_d$
for a hypersurface $F_{\Lambda}$ of degree $d$ in some
$(r+1)$-dimensional linear subspace $\Lambda$ of $\p^n$. Since
$I_\X$ is $d$-regular by Gotzmann's regularity theorem, that is, $(I_\X)_t=(I_{F_{\Lambda}})_t$ for every $t\ge d$, $I_\X$ has to be an ideal of $F_{\Lambda}$.

\smallskip

(b)$\Rightarrow$(a)
This is immediate since the Hilbert polynomial of a hypersurface of degree $d$ in some $(r+1)$-dimensional linear space  of $\p^n$ is of the form:
$$
P_\X(z)=\binom{r+z}{z}+\cdots+\binom{r+(z-d+1)}{(z-d+1)}.
$$

(a)$\Rightarrow$(c)  Recall that for a general linear space $\Lambda_i$ of dimension $n-i$ in $\p^n$
$$
\begin{array}{lllllllllllllll}
G_i(\X)  :=  G(\X\cap\Lambda_i) \\
C_i(\X)  :=  C_0(\X\cap\Lambda_i)
\end{array}
$$
for every $i=1,\dots,r$. If we apply equations~\eqref{EQ:202} and \eqref{EQ:316}  inductively, we see that
$$
 P_{\X\cap\Lambda_i}(z)
= \ds \binom{(r-i)+z}{z}+\cdots+\binom{(r-i)+(z-d+1)}{(z-d+1)}
$$
for every $i=1,\dots,r$, and thus
$$
\begin{array}{lllllllllllllllllll}
C_0(\X)=C_1(\X)=\cdots=C_{r-1}(\X)=0, \quad \text{and} \\
G(\X)=G_1(\X)=\cdots=G_{r-1}(\X)=\deg(\X)=d.
\end{array}
$$
Hence, by Theorem~\ref{T:205}, it is obvious that
$$
p_a(\X\cap \Lambda_i)=\binom{d-1}{r-i+1}
$$
for $0\leq i \leq r-1$.

\smallskip

(c)$\Rightarrow$(a)  We will show this by induction on $\dim (\X)=r$.
If $r=1$, then $i=0$, and so
$$
p_a(\X)= \binom{d-1}{2}.
$$
Thus $\X$ is a plane curve, and hence $G(\X)=\deg(\X)=d$ (see Example~\ref{Ex307}~b)).

\smallskip

Now suppose $r>1$. Let $\Y$ be a general hyperplane
section of $\X$. Then, $\dim (\Y)=r-1$ and $\Y$ satisfies the given condition. Thus, by  induction on $r$,  we have that $G(\Y)=\deg (\Y)=d$. It follows from Theorem \ref{T:204} b) and Remark~\ref{R:205} that
$$
G_i(\Y)=G_{i+1}(\X)
$$
for $0\leq i\leq r-1$. Moreover, since $G_i(\Y)=G(\Y)=\deg(\Y)=d$  by induction on
$r$, we see that
$$
G_{i+1}(\X)=G_i(\Y)=d
$$
for $0\leq i\leq r-1$. Therefore, by Theorem~\ref{T:205} (c) with the condition
$p_a(\X)={\binom{d-1}{r+1}}$, we have
$$
C_0(\X)=0,
$$
i.e.,
$$
\begin{array}{llllllllllllll}
G(\X)
& = & G(\Y)+C_0(\X) \quad (\hbox{see }\text{Theorem~\ref{T:204}}~b))\\[.5ex]
& = & G(\Y)\\[.5ex]
& = & \deg (\X),
\end{array}
$$
as we desired.
\end{proof}

\section{Gotzmann Coefficients of Reduced Equidimensional  Schemes}\label{M_R}

In this section, we prove that if $\X$ is a reduced
equi-dimensional scheme (or integral scheme) of dimension $r$ in $\p^n$ which is not a
hypersurface in some proper linear subspace, then $M(\X)$ is equal to
$G(\X)$. (As a consequence we get a characterization of when $M(\X) = 1$.)  Using this result, we can also prove Theorem~\ref{T:404} and
Corollary~\ref{Gotz_partition of integral}, which provide a
necessary condition for a numerical polynomial to be the Hilbert
polynomial of a reduced equi-dimensional scheme $\X$ in $\p^n$. In fact, what we prove is that none of the Gotzmann coefficients of such a scheme vanish, i.e.,
$$
C_\ell(\X)\neq 0 \quad \text{for any $0\leq \ell \leq r$}.
$$

Before we state and prove the main result of this section (Theorem~\ref{T:404}), we recall the Gotzmann
Persistence Theorem and prove Theorem~\ref{L:401}.  They will both be used often in what follows. Notice that the proof of Theorem~\ref{L:401} requires that $k$ be an algebraically closed field.  The version of Bertini's Theorem that we are using allows us, according to \cite{Na}, to obtain the results for any characteristic.

\begin{Thm}[Gotzmann's Persistence Theorem \cite{G2}] \label{T:401}
Let $I$ be a homogeneous ideal of $R$ generated in degree $\leq d+1$ and set $A = R/I$. If
we have maximal growth for $H(A,-)$ in degree $d$, i.e,
$$
H(A,d+1) = H(A,d)^{\langle d\rangle},
$$
then
\begin{enumerate}
\item[(a)] $I$ is $d$-regular, and
\item[(b)] $H(R/I,\ell+1) = H(R/I,\ell)^{\langle \ell\rangle}$
for all $\ell \geq d$.
\end{enumerate}
\end{Thm}

\begin{Thm}\label{L:401}
Let
$$
H_\X(d)=\binom{a_d+d}{d}+\cdots+\binom{a_{\delta}+\delta}{\delta}
$$
be the $d$-binomial expansion of $H_\X(d)$. Assume that $d>0$ is an integer for which
\begin{equation}\label{eq:402}
\Delta H_\X(d) =H_\X(d)_{\mac d}\quad\textup{and}\quad C_0(\X,d)=0, (\text{i.e., } a_\delta > 0).
\end{equation}
 Then,
\begin{enumerate}
\item[(a)] if  $G(\X,d) > C_{a_d}(\X,d)$ we have $(I_\X)_{a_d}=(I_\Lambda)_{a_d}$ for some $(a_d+1)$-dimensional linear subspace $\Lambda$ in $\p^n$ containing $X$.
\item[(b)] Furthermore, if $\delta >1$, there is a homogeneous polynomial $F$ of degree $C_{a_d}(\X,d)$ such that
$$
(I_{\X})_\ell \subset\,\left({(F)+I_{\Lambda}}\right)_\ell
$$
for every $\ell\leq d$. In other words, $\left(I_{\X}/I_{\Lambda}\right)_\ell$ has a common factor
$\overline{F}$ in $R/I_{\Lambda}$ for every $\ell\leq d$.  Moreover, $F$ is reduced if $\X$ is a reduced scheme.
\end{enumerate}
\end{Thm}

\begin{proof} (a)  Recall that we are assuming that $G(\X,d)>C_{a_d}(\X,d)$. Then by
Theorem~\ref{T:202} and Corollary~\ref{C:203}, with notation as in the proof of
Lemma~\ref{L:303} (see equation~\eqref{EQ:306-1}), we obtain
\begin{equation}\label{eq8-1}
\begin{array}{lllllllllllllll}
&& H_\X(d-1) \\[1ex]
& = & \begin{cases}
      \ds\binom{a_{d}+(d-1)}{(d-1)}+ \cdots + \binom{a_{\beta+1}+\beta}{\beta}+\binom{(\alpha+1)+(\beta-1)}{(\beta-1)},&
\text{ if } \delta =1 \\[2ex]
         \ds \binom{a_{d}+(d-1)}{(d-1)}+ \cdots +\binom{a_{\delta}+(\delta-1)}{(\delta-1)}, & \text {if } \delta
           >1.
           \end{cases}
\end{array}
\end{equation}

As we have done before, using Theorem~\ref{T:202}, Corollary~\ref{C:203}, and
Lemma~\ref{L:303}, we can obtain $H_\X(t)$ for every $t\le d-2$ from $H_\X(d-1)$ inductively. In particular, the Hilbert function in degree $d_0=C_{a_d}(\X,d)+1 \leq d$ is of the form
\begin{equation}\label{equation1}
H_X(d_0)=\binom{a_d+d_0}{d_0}+\cdots+\binom{a_d+2}{2}+\binom{\gamma+1}{1}
\end{equation}
for some $1\leq \gamma < a_d$. Moreover, by Lemma~\ref{L:303} again, we have
$$
\Delta H_\X(d_0)=H_\X(d_0)_{\langle d_0\rangle},
$$
that is,
$$
\begin{array}{lllllllllllllll}
H_\X(d_0)-H_\X(d_0-1)
& = &  \Delta H_\X(d_0) \\[.5ex]
& = & H_\X(d_0)_{\langle d_0\rangle} \qquad (  \text{by assumption})\\[1ex]
& = & \ds \left[\binom{a_d+d_0}{d_0}+\cdots+\binom{a_d+2}{2}+\binom{\gamma+1}{1}\right]_{\langle d_0\rangle}\\[2ex]
& = & \ds \binom{a_d+(d_0-1)}{d_0}+\cdots+\binom{a_d+1}{2}+\binom{\gamma}{1},
\end{array}
$$
and thus
\begin{equation} \label{EQ:403}
\begin{array}{llllllllllllllllllllll}
&   & H_\X(d_0-1) \\[1ex]
& = & H_\X(d_0)-\ds \left[\binom{(d_0-1)+a_d}{d_0}+\cdots+\binom{1+a_d}{2}+\binom{\gamma}{1}\right] \\[2ex]
& = & \ds \left[\binom{d_0+a_d}{d_0}+\cdots+\binom{2+a_d}{2}+\binom{1+\gamma}{1}\right]\\[2ex]
&   & -\ds \left[\binom{(d_0-1)+a_d}{d_0}+\cdots+\binom{1+a_d}{2}+\binom{\gamma}{1}\right] \\[2ex]
& = & \ds \binom{(d_0-1)+a_d}{(d_0-1)}+\cdots+\binom{1+a_d}{1}+\binom{\gamma}{0}\\[2ex]
& = & \ds \binom{(d_0-1)+a_d}{(d_0-1)}+\cdots+\binom{2+a_d}{1} \\[2ex]
& = & \ds \binom{(a_d+1)+(d_0-1)}{(d_0-1)}.
\end{array}
\end{equation}

Note that $\Delta H_\X({d_0-1})=H_\X(d_0-1)_{\mac {d_0-1}}$ by Lemma~\ref{L:303}. It follows from equation~\eqref{EQ:403} and Proposition~\ref{P:307} that
$$
(I_\X)_{d_0-1}=(I_\Lambda)_{d_0-1} \ \Rightarrow \
(I_\X)_{a_d}=(I_\Lambda)_{a_d}
$$
for some $(a_d+1)$-dimensional linear subspace $\Lambda$ in $\p^n$ containing $\X$.

\smallskip

(b) Now suppose $\delta>1$. Then, by Theorem~\ref{T:202}
$$
H_\X(d)=H_\X(d-1)^{\langle  d-1 \rangle},
$$
which means that the Hilbert function of $\X$  has maximal growth in degrees $d-1$ and $d$. Moreover, by Theorem~\ref{T:401}  (see also Lemma 1.4 in \cite{BGM}), the ideal
$(({I_\X})_{\leq {d-1}})$ is saturated, and thus the ideal
$$
(({I_\X})_{\leq d-1})=I_\Y
$$
for a closed subscheme $\Y \subset \p^n$. Note that $(I_\X)_\ell=(I_\Y)_\ell$ for every
$\ell\le d-1$, and
\begin{equation}\label{EQ:404}
H_\Y(\ell+1)=H_\Y(\ell)^{\langle \ell\rangle}, \quad \forall \ell\ge d-1.
\end{equation}
In particular, since
$$
\begin{array}{llllllllllllllll}
H_\Y(d)
& = & H_\Y(d-1)^{\langle d-1\rangle} \\[.5ex]
& = & H_\X(d-1)^{\langle d-1\rangle} \quad (\text{since } H_\X(d-1)=H_\Y(d-1))\\[.5ex]
& = & H_\X(d) \hskip 1.71true cm ( \text{by equation~\eqref{eq8-1}})\\[.5ex]
& = & \ds \binom{a_d+d}{d}+\cdots+\binom{a_\delta+\delta}{\delta},
\end{array}
$$
we have that by equation~\eqref{EQ:404}
$$
H_\Y(\ell)=\ds \binom{a_d+\ell}{\ell}+\binom{a_{d-1}+(\ell-1)}{(\ell-1)}+\cdots+\binom{a_\delta+(\ell-d+\delta)}{(\ell-d+\delta)}
$$
for every $\ell\ge d$, and hence
$$
\dim(\Y)=a_d \quad \text{and} \quad \deg(\Y)=C_{a_d}(\X,d).
$$

Note that by (a),
$(I_\Lambda)_{\ell}\subseteq(I_\X)_\ell=(I_\Y)_\ell$ for every
$\ell\le d$. Hence we see that $\Y$ is contained in a linear
subspace $\Lambda$ of dimension $a_d+1$.

Now let $S:=k[x_0,x_1,\dots,x_n]/I_\Lambda=R/I_\Lambda\simeq k[x_0,x_1,\dots,
x_{a_d+1}]$ and
$$
\bar I_\X=I_\X/I_\Lambda \quad \text{and} \quad \bar I_\Y=I_\Y/I_\Lambda.
$$
Considering the unmixed part of $\Y$ in Proj$(S)$, its defining
ideal has to be  a principal ideal in $S$. Since $\codim(S/\bar
I_\Y)=1$ and $\deg(\Y)=C_{a_d}(\X,d)$ we see that there is a
polynomial $F\in R$ of degree $C_{a_d}(\X,d)$, such that
$$
I_\Y\subseteq I_\Lambda +(F).
$$
Furthermore, since $(I_\X)_\ell=(I_\Y)_\ell$ for every $\ell\le d$,
we have
$$
(I_\X)_\ell\subseteq (I_\Lambda +(F))_\ell
$$
for such $\ell$, as we wished.
\end{proof}

\begin{Remk}
Because of condition \ref{eq:402} in Theorem~\ref{L:401} we have that, for every $k\leq d$, the value of $H_\X(k)$   is
completely determined by $H_\X(d)$. If there exists an integer $k$, with
$C_{a_d}(\X,d)<k\leq d$, such that the $k$-binomial
expansion of $H_\X(k)$,
$$
H_\X(k)=\binom{b_k+k}{k}+\cdots+\binom{b_{\delta(k)}+\delta(k)}{\delta(k)}
$$
has $\delta(k)>1$ then, by Theorem~\ref{L:401}, we see that
$(I_\X)_{\leq k}$ is contained in an ideal generated by $n-a_d-1$
linear forms and a homogeneous polynomial of degree
$C_{a_d}(\X,d)$. Since we know how to obtain $H_\X(k)$ from
$H_\X(d)$ for all $k\leq d$, we can check the fact that
$\delta(d-i+1)>1$ if $a_i\leq a_{i-1}+1$ and $a_{i+1}\leq a_{i}+1$
for some $0\leq i \leq d-C_{a_d}(\X,d)+1$.
\end{Remk}

%

\begin{Coro}\label{coro2}
Let $\X$ be a reduced equi-dimensional subscheme of dimension $r$
in $\p^n$. Let $$
H_\X(d)=\binom{a_d+d}{d}+\cdots+\binom{a_{\delta}+\delta}{\delta}.
$$ be the $d$-binomial expansion of $H_\X(d)$ and suppose that $a_d=r$.
If $$ \Delta H_\X(d) =H_\X(d)_{\mac d}, \quad \hbox{and} \quad C_0(\X,d)=
0
$$
for some positive integer $d$, then $\X$ is a hypersurface in a
linear subspace $\Lambda\subset\p^n$.
\end{Coro}

\begin{proof}
By Theorem~\ref{L:401}, $\X$ is contained in an $(r+1)$-dimensional
linear subspace $\Lambda$ in $\p^n$. If $\X$ is a reduced
equi-dimensional subscheme, then
$$
I_\X/I_\Lambda=\bar \wp_1\cap \cdots \cap \bar \wp_\ell
$$
where $\bar \wp_i$ is a prime ideal in $R/I_\Lambda$ of height one for every $i=1,\dots,\ell$, and thus
$$
\bar\wp_i=(\bar F_i), \text{ for some } F_i \in R,    \text{ for all }  i=1,\dots,\ell,
$$
where $F_i$ is an irreducible polynomial in $R$ for such an $i$. In other words,
$$
I_\X/I_\Lambda=(\overline{F_1\cdots F_\ell}),
$$
which means that $\X$ is a hypersurface  in $\Lambda\subset \p^n$.
\end{proof}

\begin{Coro}\label{C:403} Let $\X$ be a reduced equi-dimensional closed subscheme in $\p^n$. Then, either $G(\X)=\deg(\X)$ or $G(\X)=M(\X)$.
\end{Coro}

\begin{proof}
Let $d=G(\X)-1$. If $M(\X)<G(\X)$, then $\Delta H_\X(d)=H_\X(d)_{\mac
d}$. Furthermore, by equation~\eqref{EQ:313}, we see that
$C_0(\X,d)=0$, and so $G(\X)=\deg(\X)$ by Corollary~\ref{coro2}, as
we wished.
\end{proof}

As an immediate corollary of this, we get (for reduced and equidimensional closed subschemes of $\p^n$) a characterization of the equality $M(\X) = 1$.

\begin{Coro}\label{M(X)=1}
Let $\X$ be a reduced equi-dimensional closed subscheme of $\p^n$.  $M(\X) = 1$ if and only if $\X$ is a hypersurface in a linear subspace of $\p^n$.

\begin{proof} From Corollary~\ref{C:403} we obtain that either $M(\X) = G(\X) = 1$ or $G(\X) = \deg(\X)$.  In the first case $\X$ is a linear subspace of $\p^n$ (see Theorem~\ref{P:307}).  The second case was characterized in Theorem~\ref{T:311}, and is precisely the assertion of the Corollary.
\end{proof}
\end{Coro}

\begin{Thm}\label{T:404}
Let $\X$ be a reduced equi-dimensional closed subscheme in $\p^n$. If $\X$ is not a hypersurface in a linear subspace (i.e., $G(\X)\neq \deg(\X)$) then
$$
C_\ell(\X)\neq 0
$$
for every  $0\leq \ell \leq r$.
\end{Thm}

\begin{proof}  Let $r=\dim(\X)$. First of all, note that
$$
C_r(\X)=\deg(\X)\ne 0.
$$

Moreover, by Corollary~\ref{C:403} if $G(\X)\ne M(\X)$,  $\X$ cannot be a reduced equi-dimensional subscheme in $\p^n$ since  $G(\X)\neq \deg(\X)$, which is a contradiction. In other words,
$$
G(\X) = M(\X).
$$

\smallskip

If $C_0(\X)=0$, then $1=M(\X)< G(\X)$ by
Lemma~\ref{L:303}, which is also a contradiction. Hence $C_0(\X)\ne 0$.

\smallskip

Now suppose $C_\ell(\X)=0$ for some $0<\ell<r$. Since $\X$ is a reduced
equi-dimensional closed subscheme in $\p^n$,  for a $(n-\ell)$-dimensional general linear subspace $\Lambda_{\ell}$ in $\p^n$,
$$
\Y:=\X\cap\Lambda_{\ell}
$$
is also a reduced equi-dimensional closed subscheme in $\p^n$  by Bertini's Theorem. But then, by Remark~\ref{R:205},
$$
C_0(\Y)=C_0(\X\cap\Lambda_{\ell})=C_{\ell}(\X)=0,
$$
and
$$
\dim(\Y)=\dim(\X)-\ell=r-\ell.
$$

If $G(\Y)\ne \deg(\Y)$, then, by the same argument as above, $C_0(\Y)\ne 0$, and thus
\begin{equation}\label{EQ:405}
G(\Y)=\deg(\Y).
\end{equation}
By Theorem~\ref{T:311}, we see that $\Y$ is a hypersurface
contained in a $(r-\ell+1)$-dimensional linear subspace $\Lambda$
in $\Lambda_\ell\subset \p^n$. Since $\X$ is a reduced
equi-dimensional closed subscheme and $\Y=\X\cap\Lambda_{\ell}$, $\X$
must be contained in an $(r+1)$-dimensional linear subspace
$\Lambda_{r+1}$ in $\p^n$. Hence $\X$ is also a hypersurface
contained in $\Lambda_{r+1}$ and thus, by Theorem~\ref{T:311},
$$
G(\X)=\deg(\X),
$$
a contradiction. Therefore,
$$
C_\ell(\X)\ne 0
$$
for every $0\le \ell \le r$, as we wished.

%
%
\end{proof}

\begin{Coro}\label{Gotz_partition of integral}
Let $\X$ be a non-degenerate reduced equi-dimensional closed subscheme of
codimension $\geq 2$ in $\p^n$. Then, $C_\ell(\X)\neq 0$ for every $0 \leq
\ell \leq \dim(\X)$.
\end{Coro}

\begin{proof}
Since $\X$ is a non-degenerate closed subscheme of codimension $\geq 2$,
$G(\X)\neq \deg(\X)$ by Theorem~\ref{T:311}, and hence the statement immediately follows from Theorem~\ref{T:404}.
\end{proof}

%
%
%
%

\section{Gotzmann Coefficients for Schemes Containing Points with UPP}\label{GN_of_UPP}

In this section, we will give another situation for which the Hilbert
polynomial of a projective scheme has non-vanishing Gotzmann
coefficients. These results give a partial answer to the
conjecture proposed by Bigatti-Geramita-Migliore (Conjecture 4.9 in \cite{BGM}).
Note that a reduced, finite set of points $\Z$ is said to have the
Uniform Position Property (UPP) if for any  subset $\Y$ of $\Z$ having
cardinality $r$  we have
$$
H_\Y(\ell)=\min\{H_\Z(\ell),r\} \hbox{ for all } \ell .
$$

In \cite{BGM}, the authors showed how the imposition of uniform
position on a set of points is reflected in both the ideal of the points and in the values of the Hilbert function of the points.  For example,

\begin{Lem}[Lemma 4.4 \cite{BGM}] \label{lem4}
Let $\Z$ be a finite  reduced set of points in $\p^n$ with UPP and
suppose that the forms in $(I_{Z})_d$ have a common factor $F$. Then
$F$ is irreducible and $((I_\Z)_{\leq d})=(F)$.
\end{Lem}

The next lemma will be used to prove Theorem~\ref{T:504} and also to illustrate a property of schemes for which some Gotzmann coefficient is 0.

\begin{Lem}\label{lem3}
Let $I$ be a homogeneous ideal in $R=k[x_0,x_1,\dots,x_n]$ such that $\reg (I)< d$. For
general linear forms $L_1, \ldots, L_\ell \in R$, suppose that all elements of
$$
\frac{I_d+(L_1,\ldots,L_\ell)_d}{(L_1,\ldots,L_\ell)_d}
$$
have a common factor of positive degree in $R/(L_1,\ldots,L_\ell)$.
Then  all elements of $I_d$ also have a common factor of positive degree in $R$.
\end{Lem}

\begin{proof} Let $J$ be a homogeneous  ideal of
$R/(L_1,\ldots,L_{\ell-1})$ defined by
$$
J=\frac{I+(L_1,\ldots,L_{\ell-1})}{(L_1,\ldots,L_{\ell-1})}.
$$
We let $\overline{L}_\ell$ denote the image of
$L_\ell$ in $R/(L_1,\ldots,L_{\ell-1})$. Since
$$
\frac{I_d+(L_1,\ldots,L_\ell)}{(L_1,\ldots,L_\ell)}\cong\frac{J_d+(\overline{L}_\ell)}{(\overline{L}_\ell)},
\quad \text{and} \quad
\reg(J)\leq \reg(I)<d,
$$
it is enough to consider the case
$\ell=1$ by induction on $\ell$.

\smallskip

Now suppose $\ell=1$ and  $I$ is a saturated ideal. Then we may assume that a general linear form $L_1$ is a non-zero divisor of $R/I$. Let $S:=R/(L_1)$. Then  we have
$$
\dim(R/I)=\dim(S/J)+1.
$$
Moreover, since the ideal $J$ has a common factor in $S$, there exists an  irreducible polynomial $F$ in $R$ such that $J \subset (\overline F)$. This means that
$$
n-1 = \dim(S/(\overline F)) \leq \dim(S/J) =\dim (R/I)-1 \le n-1 \qquad (\text{since }
{\rm height}(I)\ge 1).
$$
Hence, $\dim(R/I)= \dim(S/J)+1= n$, and so there is an associated prime $P$ of $I$ such that $\dim(R/P)=n$.  Furthermore, since $\dim(R/P)=n$,
$$
P=(G)
$$
for some irreducible polynomial $G$ in $R$. In other words,  $I \subset P=(G)$, and thus $I$ has a common divisor $G$.

\smallskip

For a homogeneous ideal $I$ in $R$, note that
$$
I_m=(I^{\sat})_m
$$
for a sufficiently large $m\gg 0$ and the saturation degree is $\le \reg(I)$. Therefore
$$
\left(\frac{I+(L)}{(L)}\right)_m=\left(\frac{I^{\sat}+(L)}{(L)}\right)_m,
$$
for such $m$, and thus $I_d$ has a common factor as above.  This completes the proof.
\end{proof}

\begin{Thm}\label{T:504}
Let $\Z$ be a finite non-degenerate reduced set of points in $\p^n$
with UPP and suppose that $0\ne ((I_\Z)_{\leq d})$ is saturated for
some $d>0$. If $\X$ is the closed subscheme defined by $((I_\Z)_{\leq
d})$ in $\p^n$ then, either
\begin{itemize}
  \item [(1)] $I_\X=(F)$ for some irreducible polynomial $F$, or
  \item [(2)]  $C_\ell(\X)\neq 0$  for every $0 \leq \ell \leq \dim(\X)$.
\end{itemize}
\end{Thm}

\smallskip

\begin{proof} First of all, note that
\begin{itemize}
\item[(1)] $\Delta H_\X(s)=H_\X(s)_{\langle s\rangle}$ for sufficiently large
$s\gg 0$,
\item[(2)]  $C_{\dim(\X)}(\X)=\deg (\X) \ne 0$.
\end{itemize}

\smallskip
For sufficiently large $s\gg 0$, let
$$
H_\X(s)=\binom{a_s+s}{s}+\cdots+\binom{a_\delta+\delta}{\delta}
$$
be the $s$-binomial expansion of $H_\X(s)$. Then it is obvious that
$$
\Delta H_\X(s)=H_\X(s)_{\langle s\rangle}
$$
for such an $s$.

Suppose that $C_i(\X)=0$ for some $0\leq i < \dim(\X)$.

\medskip
\noindent{\em Case 1:} $C_0(\X) = 0$.
\medskip

We separate this case into two subcases, according as $G(\X) > C_{a_s}(\X) = \deg (\X)$ or $G(\X) = \deg(\X)$.  We first consider $G(\X) > \deg(\X)$.

Then by Theorem~\ref{L:401} a), $(I_\X)_{a_s} = (I_\Lambda)_{a_s}$ for some $(a_s + 1)$-dimensional linear space $\Lambda$.  But, since $\X \supset \Z$ and $\Z$ is a non-degenerate set of points, $\Lambda = \p^n$.  So, $\dim \Lambda = a_s + 1 = n$.  Thus, $\dim \X = a_s = n-1$ and so $\X$ is a hypersurface in $\p^n$.  By Lemma~\ref{lem4}, $F$ is irreducible and has degree $\leq d$ and $I_\X = (F)$ and we are done in this subcase.

Let's now suppose that $G(\X) = C_{a_s}(\X) = \deg \X$.  By Theorem~\ref{T:309} we have $(I_\X)_s = J_s$ for all $s \gg 0$, where $J = (L_1, \ldots , L_{n-(a_s + 1)}, F)$ and where $\ell = \deg F = C_{a_s}(\X)$.  Since the zeroes of $J$ contain $\Z$, there can be no linear forms in $I_\X$ and hence $(I_\X)_s = (F)_s$ for all $s \gg 0$.  But this implies that $I_\X = (F)$ and we are done by Lemma~\ref{lem4}.  That completes this subcase and finishes the case in which $C_0(\X)$ is the Gotzmann coefficient which is 0.

\medskip\noindent{\em Case 2:} $C_\ell(\X) = 0$ for some $0 < \ell < \dim \X$.

\medskip

Then, by Remark~\ref{R:205}, for a general linear subspace $\Lambda_\ell$ in $\p^n$ of dimension $n-\ell$,
$$
C_0(\X\cap \Lambda_\ell)=C_\ell(\X)=0.
$$
Since we are assuming that $s \gg 0$ we can assume that $\delta >1$ and so we can apply Theorem~\ref{L:401} b) (using Remark~\ref{R:205}) to $\X \cap \Lambda_\ell$.

Thus, by Theorem~\ref{L:401} b) and Lemma~\ref{lem3} (and since $s \gg 0$) we obtain that $(I_\X)_s$ has
  a common factor $F$.  By Lemma~\ref{lem4} $F$ is irreducible and
$$
\begin{array}{lllllllllllllllllll}
& (F)=((I_\X)_{\le s})=(((I_\Z)_{\leq d})_{\le s})=((I_\Z)_{\leq d})= I_\X\\
\Leftrightarrow
& G(\X)=\deg(\X) \quad (\text{by Theorem~\ref{T:311}}),
\end{array}
$$
which completes the proof.
\end{proof}

%
%
%
%
%
%
%

Before we finish this section, we prove a lemma which will be used
for the proof of Theorem~\ref{T:506}. Recall that although our definition of
{\it persistence index} (see Definition~\ref{perindex}) was for any quotient ring of $R = k[x_0, \ldots , x_n]$, we only gave information on it in case $A = R/I$ when $I = I_\X$ was the ideal of a closed subschemes of $\p^n$, i.e., only for saturated ideals.   Our next lemma calculates the persistence index in general.

\begin{Lem}\label{L:501}
Let $I$ be a homogeneous ideal of $R=k[x_0,\ldots,x_n]$. Then,
\[G(R/I)=\max\{G(R/I^{\sat}),\sat(I)\}.\]
\end{Lem}
\begin{proof}
Suppose that $d\geq G(R/I)$. Then, by definition,
$$H(R/I,t+1) = H(R/I,t)^{\mac t}$$
for all $t\geq d$. Consider the lex-segment ideal $I^{\lex}$ of
$I$. Since we have the maximal growth of Hilbert function in
degrees greater than $d$, $I^{\lex}$ does not have monomial
generators whose degree is larger than $d$, i.e., for all $t> d$,
$$\beta_{0,t}(I^{\lex})=0.$$
From the result given by Bigatti, Hulett, and Pardue in \cite{Bi, Hu, Pa}, we have that
$\beta_{p,j}(I)\leq \beta_{p,j}(I^{\lex})$ for all $p,j$. Hence it
follows that $I$ is $d$-regular from Theorem~\ref{T:401}, i.e.,
$d\geq \reg(I)$. On the other hand, we know that, by Proposition
2.6 in \cite{G2},
$$\reg(I)=\max\{\reg(I^{\sat}), \sat(I)\}.$$
Hence, for all $t\geq d$,
$$I_t=I^{\sat}_t$$
and so
\[
\begin{array}{llllllllllllllllllllll}
H(R/I^{\sat},t)^{\mac t}& = & H(R/I,t)^{\mac t}\quad (\hbox{since } t\geq\sat(I))\\[1ex]
                 & = & H(R/I,t+1)\quad (\hbox{since } t\geq G(R/I))\\[1ex]
                 & = & H(R/I^{\sat},t+1) \quad ( \hbox{ because } t\geq \sat(I)),
\end{array}
\]
and this means $d\geq \max\{G(R/I^{\sat}), \sat(I))\}$.

\smallskip

Conversely, suppose that $d\geq \max\{G(R/I^{\sat}), \sat(I)\}$.
Then, for all $t\geq d$,
\[
\begin{array}{llllllllllllllllllllll}
H(R/I,t+1) & = & H(R/I^{\sat},t+1) \quad (\hbox{since } t\geq \sat(I))\\[1ex]
           & = & H(R/I^{\sat},t)^{\mac t} \quad (\hbox{since } t\geq G(R/I^{\sat}))\\[1ex]
           & = & H(R/I,t)^{\mac t}\quad (\hbox{ because } t\geq \sat(I)).
\end{array}
\]
Hence we obtain that $d\geq G(R/I)$, as we wished.
\end{proof}

 Let $h$ and $d$ be positive
integers. For the $d$-binomial expansion of
$$
h=\binom{a_d+d}{d}+\cdots+\binom{a_{\delta}+\delta}{\delta},
$$
and for some $0\leq i \leq a_d$, we define
$$
C_i(h,d)=|\{\ell\,|\,a_\ell =i\}|.
$$
Let $Z$ be a non-degenerate finite reduced  set of   points $\mathbb P^n$ with
UPP. The following theorem says what happens if the Hilbert function  of $Z$ has a maximal growth in degree $d$.

\begin{Thm}\label{T:506}
Let $\Z$ be a non-degenerate finite reduced  set of  points with
UPP in $\p^n$ and let  $\Delta {\bf
H}=(h_0,h_1,\ldots,h_t)$ be the first difference of the Hilbert
function $\bf H$ of $\Z$. Suppose that
$$
h_d=\binom{a_d+d}{d}+\cdots+\binom{a_{\delta}+\delta}{\delta},
$$
and that $h_{d+1}=h_d^{\langle d\rangle }$ for some $d$. Let $(I_\Z)_{\leq d} = (I_\X)_{\leq d}$.

Then, either $X$ is an irreducible hypersurface or
$$
C_\ell(h_d,d)\neq 0
$$
for every $0\leq \ell \leq a_d$.
\end{Thm}

\begin{proof} First, note that $I=((I_\Z)_{\leq d})$ is a saturated ideal defining a closed subscheme $\X$ since $h_{d+1}=h_d^{\langle d\rangle }$. Let $L$ be a general linear form in $R=k[x_0,x_1,\dots,x_n]$ and
$\displaystyle{J=(I+(L))/{(L)}}.$ If $H$ is the hyperplane defined by the vanishing of $L$, then $J^{\sat}$
is the defining ideal of $\X\cap H$ in $S=R/(L)$. Moreover, since
$$
h_{d+1}=h_d^{\langle d\rangle } \quad\text{and}\quad H(R/(I+(L)),s)=h_s
$$
for every $s \ge d$, we see  that, by Theorem~\ref{T:401} (Gotzmann's Persistence Theorem),
$$
H(S/J,s+1)=H(S/J,s)^{\langle s\rangle}
$$
for such $s$. Note that, by Lemma~\ref{L:501},
$$
G(S/J)=\max\{\, G(S/J^{\sat}),\, \sat(J)\,\}.
$$
Hence we have
\begin{equation}\label{Meq1}
d\geq G(S/J)\geq G(X\cap H)=G(S/J^{\sat}).
\end{equation}
From the $d$-binomial expansion of $h_d$, given above,
we have $C_\ell(h_d,d)=C_\ell(\X\cap H)$ for every $0\leq \ell\leq a_d$ since
$d\geq G(\X\cap H)$. Thus $C_\ell(h_d,d)=C_{\ell+1}(\X)$ for every
$0\leq \ell\leq a_d$ (see Theorem \ref{T:204}). Therefore it follows from
Theorem~\ref{T:504} that
$$
C_\ell(h_d,d)\neq 0
$$
for $0\leq \ell \leq a_d$ if $\X$ is not a hypersurface.
\end{proof}

\begin{Remk}\label{Hil_not_UPP}
Theorem~\ref{T:506} gives us a condition on the Hilbert function
that prohibits the existence of points with UPP from having that Hilbert
function. If we have maximal growth of the $h$-vector of a finite reduced set of
points $\Z$ in degree $d$ such that $C_i(h_d, d)= 0$ and
$C_j(h_d, d)\neq 0$ for some $0\leq i\neq j< \dim(\X)$ (where $\X$ is a closed subscheme in $\p^n$ defined by the saturated ideal $((I_\Z)_{\leq d})$) then $\Z$
cannot have UPP .
\end{Remk}

\begin{Ex}\label{Example_UPP}
\begin{itemize}
  \item[(a)] Let $\Z$ be a non-degenerate set of points in $\p^5$ and let
$\Delta{\bf H_\Z}=(1,5,12,22,37,57,82,112,147)$ be the first difference of the Hilbert function ${\bf H}_\Z$ of $\Z$. Then
$h_7$ and $h_8$ of $\Delta{\bf H_\Z}$ have a maximal growth in degrees $7$ and $8$. Moreover, since
$$
\begin{array}{lllllllllllllll}
h_7 & = & 112 \\[1ex]
    & = & \ds \binom{2+7}{7} +  \binom{2+6}{6} +\binom{2+5}{5} +\binom{2+4}{4} +\binom{2+3}{3} +\binom{2}{2} +\binom{1}{1},\\[3ex]
\end{array}
$$
we have
$$
C_2(h_7,7)=5, \quad C_1(h_7,7)=0, \quad \text{and} \quad C_0(h_7,7)=2.
$$
However, the saturated ideal $((I_\Z)_{\le 7})$ cannot define an irreducible hypersurface and $C_1(h_7,7)=0$, i.e., $\Z$ cannot have UPP.
  \item[(b)] Let $\Z$ be a set of non-degenerate reduced points in $\p^n$
($n\geq 3$) with $h$-vector ($1,h_1,\ldots,h_t$).  Let $d$ be an integer such that $0<d<t$ and suppose that $\Z$ does not lie on any hypersurfaces of degree $d-1$.  Let $h_d$ be given by
$$
h_d=\binom{2+d}{d}+\cdots+\binom{2+(d-\gamma+1)}{(d-\gamma+1)}+\binom{d-\gamma}{d-\gamma}+\cdots+\binom{d-\ell+1}{d-\ell+1}
$$
for some $\ell <\gamma <d$ and define $h_{d+1}:=h_d^{\langle d\rangle }$.    Then $\Z$ cannot have UPP since the saturated ideal
$((I_\Z)_{\leq d})$ does not define an irreducible hypersurface (there is more than one form in $(I_\Z)_d$ and none in degree $d-1$) but $C_1(h_d,d)=0$.
\end{itemize}

\medskip

One can easily construct lots of other examples of $h$-vectors which are not the $h$-vectors of points with UPP.
\end{Ex}


\begin{thebibliography}{10}

\bibitem{AC}
J.~Ahn and Y.~Cho.
\newblock Gotzmann numbers of graded {$k$}-algebras and {B}etti numbers of the
  associated lex-segment ideal.
\newblock {\em Comm. Algebra}, 33(1):301--317, 2005.

\bibitem{Bi}
A.~Bigatti.
\newblock Upper bounds for the {B}etti numbers of a given {H}ilbert function.
\newblock {\em Comm. Algebra}, 21(7):2317--2334, 1993.

\bibitem{BGM}
A.~Bigatti, A.V. Geramita, and J.C. Migliore.
\newblock Geometric consequences of extremal behavior in a theorem of
  {M}acaulay.
\newblock {\em Trans. Amer. Math. Soc.}, 346(1):203--235, 1994.

\bibitem{EH1}
D.~Eisenbud and J.~Harris.
\newblock {\em Curves in Projective Space}.
\newblock Les Presses de L'Universit$\acute{e}$ de Montr$\acute{e}$al. 1982.

\bibitem{GM-1}
A.V. Geramita and J.C. Migliore.
\newblock Hyperplane sections of a smooth curve in {${\bf P}\sp 3$}.
\newblock {\em Comm. Algebra}, 17(12):3129--3164, 1989.

\bibitem{Go}
G.~Gotzmann.
\newblock Eine {B}edingung f\"ur die {F}lachheit und das {H}ilbertpolynom eines
  graduierten {R}inges.
\newblock {\em Math. Z.}, 158(1):61--70, 1978.

\bibitem{G1}
M.~Green.
\newblock Restrictions of linear series to hyperplanes, and some results of
  {M}acaulay and {G}otzmann.
\newblock In {\em Algebraic curves and projective geometry ({T}rento, 1988)},
  volume 1389 of {\em Lecture Notes in Math.}, pages 76--86. Springer, Berlin,
  1989.

\bibitem{G2}
M.~Green.
\newblock Generic initial ideals.
\newblock In {\em Six lectures on commutative algebra ({B}ellaterra, 1996)},
  volume 166 of {\em Progr. Math.}, pages 119--186. Birkh\"auser, Basel, 1998.

\bibitem{Hu}
H.A. Hulett.
\newblock Maximum {B}etti numbers of homogeneous ideals with a given {H}ilbert
  function.
\newblock {\em Comm. Algebra}, 21(7):2335--2350, 1993.

\bibitem{I_K}
A.~Iarrobino and S.L. Kleiman.
\newblock The gotzmann theorems and the hilbert scheme. appendix c in power
  sums, gorenstein algebras, and determinantal loci, by a. iarrobino and v.
  kanev.
\newblock pages 289--312. Springer-Verlag, Berlin, 1999.

\bibitem{KR}
M.~Kreuzer and L.~Robbiano.
\newblock {\em Computational commutative algebra. 2}.
\newblock Springer-Verlag, Berlin, 2005.

\bibitem{M}
F.S. Macaulay.
\newblock Some properties of enumeration in the theory of modular systems.
\newblock {\em Proc. Lond. Math. Soc.}, 26(1):531--555, 1927.

\bibitem{MR}
R.~Maggioni and A.~Ragusa.
\newblock The {H}ilbert function of generic plane sections of curves of {${\bf
  P}\sp 3$}.
\newblock {\em Invent. Math.}, 91(2):253--258, 1988.

\bibitem{Mi}
J.C. Migliore.
\newblock The geometry of {H}ilbert functions.
\newblock In {\em Syzygies and {H}ilbert functions}, volume 254 of {\em Lect.
  Notes Pure Appl. Math.}, pages 179--208. Chapman \& Hall/CRC, Boca Raton, FL,
  2007.

\bibitem{Na}
Y.~Nakai.
\newblock Note on the intersection of an algebraic variety with the generic
  hyperplane.
\newblock {\em Mem. Coll. Sci. Univ. Kyoto Ser. A. Math.}, 26:185--187, 1951.

\bibitem{Pa}
K.~Pardue.
\newblock Deformation classes of graded modules and maximal {B}etti numbers.
\newblock {\em Illinois J. Math.}, 40(4):564--585, 1996.

\bibitem{Preser}
K.J. Presser.
\newblock Lex-segment ideals and the maximal growth of {H}ilbert functions.
\newblock {\em Comm. Algebra}, 30(1):237--254, 2002.

\bibitem{St}
R.P. Stanley.
\newblock Hilbert functions of graded algebras.
\newblock {\em Advances in Math.}, 28(1):57--83, 1978.

\end{thebibliography}

\end{document}